\documentclass[12pt]{amsart}

\usepackage{amsmath}
\usepackage{amsthm}
\usepackage{amsfonts}
\usepackage{amssymb}
\usepackage[all]{xy}
\usepackage{url}
\usepackage{comment}
\usepackage{eucal}

\newcommand\C{\mathbb{C}}

\newcommand\N{\mathbb{N}}
\newcommand\R{\mathbb{R}}

\newcommand\Q{\mathbb{Q}}
\newcommand\Z{\mathbb{Z}}

\newcommand\Spin{\text{Spin}}

\newcommand\Id{\textup{id}}

\newcommand\acs{\varphi}  				
\newcommand\nacs{S\varphi}				
\newcommand\bnacs{\overline{\scxs}}		
\newcommand\acxs{J}						
\newcommand\scxs{\zeta}					
\newcommand\bscxs{\bar \zeta}			
\newcommand\stein{\sigma}				

\newcommand\svb{\mu}					

\newcommand\im{\textup{Im}}
\newcommand\Ker{\textup{Ker}}
\newcommand\del{\partial}

\newcommand\wt[1]{\widetilde{#1}}

\newcommand\Coker{\textup{Coker}}
\newcommand\Aut{\textup{Aut}}

\newcommand{\xra}{\xrightarrow}

\newcommand{\dlra}[1]{\stackrel{#1}{\longrightarrow}}

\newcommand{\an}[1]{\langle{#1}{\rangle}}
\newcommand{\wh}{\widehat}

\newcommand{\cal}[1]{\mathcal{#1}}

\newtheorem{Theorem}{Theorem}[section]
\newtheorem{Definition}[Theorem]{Definition}
\newtheorem{Lemma}[Theorem]{Lemma}
\newtheorem{Proposition}[Theorem]{Proposition}

\newtheorem{Conjecture}[Theorem]{Conjecture}
\newtheorem{Corollary}[Theorem]{Corollary}

\newtheorem{Notation}[Theorem]{Notation}

\newtheorem{Problem}[Theorem]{Problem}
\newtheorem{Question}[Theorem]{Question}

\theoremstyle{remark}
\newtheorem{Example}[Theorem]{Example}
\newtheorem{Remark}[Theorem]{Remark}

\usepackage{color}

\newcommand{\jbedit}[1]{\begingroup\color{black}#1\endgroup}

\newcommand{\dcedit}[1]{\begingroup\color{black}#1\endgroup}

\advance \topmargin by -\headheight
\advance \topmargin by -\headsep
     
\evensidemargin \oddsidemargin
\date{\today}

\begin{document}
\title{The topology of Stein fillable manifolds in high dimensions I} 
\author{Jonathan Bowden}
\address{Math.\ Institut, Universit\"{a}t Augsburg, Universit\"{a}tstr 14, 86159 Augsburg, Germany}
\email{jonathan.bowden@math.uni-augsburg.de}
\author{Diarmuid Crowley}
\address{Max Planck Institut f\"ur Mathematik, Vivatsgasse 7, D-53111 Bonn, Germany}
\email{diarmuidc23@gmail.com}
\author{Andr\'{a}s I. Stipsicz}
\address{ R\'{e}nyi Institute of Mathematics, Re\'{a}ltanoda u. 13-15., Budapest, Hungary H-1053}
\email{stipsicz@renyi.hu}

\maketitle

\begin{abstract}
We give a
bordism-theoretic characterisation of those closed almost contact
$(2q{+}1)$-manifolds (with $q\geq 2$) which admit a Stein fillable
contact structure. 
Our method is to apply Eliashberg's $h$-principle for Stein manifolds in the
setting of Kreck's modified surgery. As an
application, we show that any  simply connected almost contact
7-manifold with torsion free second homotopy group is Stein fillable.  We also
discuss the Stein fillability of exotic spheres and examine
subcritical Stein fillability.
\end{abstract}

\section{Introduction} \label{sec:introduction} 
There have been several recent breakthroughs concerning the existence of contact structures in higher dimensions. 
In \cite{bourgeois, Geiges&Stipsicz10, HajdukWalczak, BCSS2}
contact structures on certain
product manifolds, in \dcedit{in \cite{Ding&Geiges2012} contact structures on certain
$S^1$-bundles were constructed} 
and Casals-Pancholi-Presas \cite{CPP12} and Etnyre~\cite{Etnyre12} have shown that every almost
contact 5-manifold is contact. The general existence question on
higher dimensional manifolds is, however, still open.  
\footnote{The recent preprint of Borman, Eliashberg and
  Murphy~\cite{BEM} offers a proof or the statement that any almost
  contact manifold is contact.}  (In the following we will assume that
all almost contact manifolds are closed --- for open manifolds the
existence question has been settled by using Gromov's $h$-principle.)

Motivated by their 3-dimensional analogues, various notions of fillability and overtwistedness of
contact structures on higher dimensional manifolds have been
extensively studied, cf.\,\,\cite{Massot12, MNPS}. The class of contact structures 
satisfying an appropriate $h$-principle (as overtwisted 3-dimensional 
contact manifolds do), however, has not yet been identified.
\footnote{Again, the recent preprint of Borman, Eliashberg and Murphy
  \cite{BEM} contains such a concept.}

In view of this, one is led to consider the existence of
contact structures with special properties, the most natural of which
is perhaps Stein fillability. Recall that a contact manifold is Stein
fillable if it can be realised as the boundary of a \emph{Stein
  domain}, which is a compact, complex manifold with boundary
admitting a strictly plurisubharmonic function for which the boundary
is a regular level set. One of the motivating problems which concerns
us here is the following:

\begin{Problem}[Stein Realisation Problem]\label{prob:SteinReal}
Determine the almost contact structures which are realised by Stein
fillable contact structures.
\end{Problem}

By Eliashberg's characterisation of Stein manifolds
\cite{Eliashberg??, Cieliebak&Eliashberg12}, the existence of Stein
fillings in higher dimensions is reduced to a topological question
about whether a given manifold admits a nullbordism containing only
handles up to the middle dimension and whose tangent bundle admits an
almost complex structure.  For a given $2n$-manifold (with $n>2$), a
direct argument can decide whether it admits a Stein structure, but it
is more delicate to see whether an odd dimensional manifold can be
presented as the boundary of \dcedit{some} manifold carrying a Stein structure.

This question can be naturally studied within the framework of
(the appropriate) cobordism theory. The main goal of this article is to
elucidate the algebro-topological consequences of the above
characterisation of Stein fillability by Eliashberg.  This approach
for constructing contact structures on manifolds was initiated by
Geiges \cite{Geiges97}, and here we pursue the same ideas, using the
general setting of Kreck's modified surgery \cite{Kreck99}.

\dcedit{
An almost contact structure $\acs$ on a closed oriented $(2q{+}1)$-manifold $M$
is a reduction of the structure group of the tangent bundle of $M$ to $U(q)$.
The almost contact structure $\acs$ defines a complex structure $\scxs$ on the stable 
normal bundle of $M$ which we regard as a map $\scxs \colon M \to BU$, where $BU$ denotes
the classifying space of stable complex vector bundles. 
The $q^{th}$ Postnikov factorisation of $\scxs$ consists of a space $B^{q-1}_\scxs$ and maps
\[  M \xra{~~\bscxs~~} B^{q-1}_\scxs \xra{~\eta^{q-1}_\scxs~} BU, \]
such that $\eta^{q-1}_\scxs$ is a fibration, $\scxs = \eta^{q-1}_\scxs \circ \bscxs$, the map
$\bscxs$ is a $q$-equivalence and $\eta^{q-1}_\scxs$ is a $q$-coequivalence;
see Definition \ref{def:complex-normal-k-type}. In addition, there is a canonical bundle isomorphism $\bscxs^*(\eta^{q-1}_\scxs) \cong \nu_M$, where $\nu_M$ is the stable normal bundle of $M$ (and $\eta ^{q-1}_\scxs$ is regarded as a stable oriented vector bundle over $B^{q-1}_\zeta$).}
The pair $(M, \bscxs)$ defines a
bordism class,
\[ [M, \bscxs] \in \Omega_{2q{+}1}(B^{q-1}_\scxs; \eta^{q-1}_\scxs), \]
in the bordism theory defined by the complex bundle $(B^{q-1}_\scxs, \eta^{q-1}_\scxs)$;
see Section \ref{subsec:surgery_settingA} and Definition \ref{def:Beta-bordism}.
With these notions in hand we have the following:
\begin{Theorem} \label{thm:main}
A closed almost contact manifold $(M, \acs)$ of dimension
$(2q{+}1)\geq5$ admits a Stein filling if and only if $[M, \bscxs] = 0
\in \Omega_{2q{+}1}(B^{q-1}_\scxs; \eta^{q-1}_\scxs)$.
\end{Theorem}

\noindent (An expanded version of the result is given in
Theorem~\ref{thm:Stein}).
The bordism groups appearing in Theorem \ref{thm:main} are isomorphic,
via the Pontrjagin-Thom isomorphism, to the stable homotopy groups of
the Thom spectrum of $\eta^{q-1}_\scxs$. Hence the entire apparatus of
stable homotopy theory is available to compute these groups,
and if one can show that $\Omega_{2q{+}1}(B^{q-1}_\scxs; \eta^{q-1}_\scxs)=0$, 
a general existence result follows.
We will show that this is the case for simply connected 7-manifolds
with torsion free second homotopy groups.

\begin{Theorem} \label{thm:seven_mfold_contact}
Let $M$ be a closed simply connected $7$-manifold with $\pi_2(M)$ torsion
free.  Then $M$ admits an almost contact structure, and every almost
contact structure on $M$ can be represented by a Stein fillable
contact structure.
\end{Theorem}
\noindent

(Theorem \ref{thm:seven_mfold_contact} can be interpreted as an extension of existence
results for contact structures on $1$-connected almost contact
5-manifolds and 2-connected 7-manifolds~\cite{Geiges92, Geiges97}).

As expected, the existence of a Stein fillable contact structure on a
manifold depends on the smooth structure it carries, and not simply on
the underlying homeomorphism type. This fact can be most transparently
demonstrated by showing that certain exotic spheres (i.e. smooth
manifolds homeomorphic but not diffeomorphic to the sphere of the same
dimension) do not carry any Stein fillable contact structures. Using
the obstruction class of Theorem~\ref{thm:main}, we prove the following
theorem which answers a question raised by Eliashberg, see \cite[3.8]{Problem12},
in roughly three-quarters of all dimensions.

\begin{Theorem} \label{thm:Not_Stein_sphere}
Let $\Sigma^{2q{+}1}$ be a homotopy sphere which admits no framing
bounding a paralellizable manifold. 
If $q \not\equiv 1, 3, 7$~mod~$8$ or if $q \equiv 1$~mod~$8$ and $q > 9$ or if $q = 7$ or $15$,
then $\Sigma ^{}$ admits no Stein fillable contact structure.
\end{Theorem}
\noindent (A more precise version of the result is given in
Theorem~\ref{thm:homotopy_sphere}.)

The obstruction for manifolds to carry Stein fillable contact
structures can also be used to establish the following extension of a
3-dimensional result found in \cite{Bowden12} to higher
dimensions. (The construction is based on non-connected examples of
exactly fillable manifolds of \cite{Massot12} which are not Stein
fillable.)

\begin{Theorem}\label{thm:exact-notStein}
  There exist connected, exactly fillable, contact manifolds that are
  not Stein fillable in all dimensions greater than three.
\end{Theorem}

Further results 
\dcedit{(largely concerning $(q-1)$-connected $(2q{+}1)$-manifolds)}
are deferred to a continuation of the present work in \cite{BCS} --- in the present
paper we emphasize the basic features of the method and restrict
ourselves to the applications listed above.

The paper is organized as follows.  In Section~\ref{sec:appendix} we
give a review of the formulation of Kreck's surgery theory, with the
necessary adaptations to the setting of contact and Stein geometry. 
In particular, we define the obstruction class $[M, \bscxs]$ of Theorem~\ref{thm:main}
in the appropriate bordism group. In
Section~\ref{sec:complex} we set up notations, 
recall some basic notions from contact and symplectic topology and 
interpret  the
obstruction class defined earlier in 
the contact context as the obstruction for an almost contact structure to be representable by a
Stein fillable contact structure. This result leads to  the topological
characterization of Stein fillability of Theorem~\ref{thm:main}.  As an application, in
Section~\ref{sec:7_manifolds} we provide the proof of
Theorem~\ref{thm:seven_mfold_contact}.
Section~\ref{sec:homotopy_spheres} concentrates on highly connected
manifolds, and (among other results) we prove
Theorem~\ref{thm:Not_Stein_sphere}. In Section~\ref{sec:further} we
discuss further obstructions for Stein fillability, and prove
Theorem~\ref{thm:exact-notStein}.  \dcedit{Finally in
Section~\ref{sec:sub-critical} we formulate a version of the Filling Theorem,
Theorem \ref{thm:sub-c_and_products}, which provides an obstruction for subcritical Stein fillability.}
We also examine the Stein fillability of the product of a contact
manifold with a 2-dimensional surface.

\bigskip

{\bf {Acknowledgements:}} The authors would like to thank the
Max-Planck-Institute in Bonn where parts of this work were carried out
and also Anna Abczynski, \dcedit{Fabian Hebestreit} and Oscar
Randal-Williams for helpful comments.  AS was partially supported by
OTKA NK81203, by the \emph{Lend\"ulet program} of the Hungarian
Academy of Sciences and by ERC LDTBud.  The present work is part of
the authors' activities within CAST, a Research Network Program of the
European Science Foundation.

\section{Complex modified surgery} 
\label{sec:appendix}
\dcedit{
In this section we develop the theory of modified surgery from
\cite{Kreck99} in the setting where all the stable vector bundles under consideration have complex 
structures.  
Before turning to the details, we describe the motivation and the heuristics of
the constructions encountered in the section.

The fundamental result underpinning our approach to proving existence results of
(Stein fillable)
contact structures is the following discovery of Eliashberg \cite{Cieliebak&Eliashberg12, Eliashberg??} (given formally in Corollary~\ref{thm:h-principle}):
A $(2q{+}2)$-dimensional compact manifold $W^{2q{+}2}$ with boundary $\partial W=M$ and $q\geq 2$ 
admits 
the structure of a Stein domain (and hence $M$ admits a Stein fillable contact structure) 
if and only if $W$ admits (a) an almost complex structure and (b) a handle decomposition
involving handles with index at most $q{+}1$. Therefore in studying the existence of 
Stein fillable contact structures on $M$ we need to understand whether it bounds 
almost complex manifolds with the required handle decompositions.


Recall that an almost contact structure on $M^{2q{+}1}$ is by definition the reduction of the 
structure group of the tangent bundle to $U(q)\subset SO(2q{+}1)$, or equivalently a map 
$M\to BU(q)$ covering the classifying map $M\to BSO(2q{+}1)$ of its tangent bundle. 
Similarly, an almost complex
structure on $W^{2q{+}2}$ with $\partial W=M$ is simply a map $W\to BU(q{+}1)$ covering the
classifying map of its tangent bundle. Since computations of bordism groups are best
done with stable bundle structures, we will consider \emph{stable} almost contact 
(and \emph{stable} almost complex) structures, hence we need to examine maps from $M$ or $W$ to $BU$.
Complex bordism theory  (see Example \ref{ex:Complex_bordism}) tells us that
every stably almost contact manifold is the boundary of a stably complex $(2q{+}2)$-manifold.
(Since Eliashberg's results requires an almost complex structure, rather than a stabilization 
of it, we need to examine how a stable structure induces an unstable one, and how unique 
this unstable structure is. This will be explained in Section~\ref{subsec:stable_and_unstable_surgery}.)

The other condition on $W$ for the existence of a Stein structure
(regarding a handle decomposition with no handle above the middle
dimension) seems less amenable to the techniques of homotopy theory.
However, a classic theorem of Wall, (formally stated as
Theorem~\ref{thm:ofwall} in Section~\ref{subsec:key_surgery_lemmas}),
turns this condition to a problem about homotopy groups.  Wall's
theorem builds on the $s$-cobordism theorem: Consider $W$
as a handlebody built on $M$, ie. turn the handle decomposition coming
from a plurisubharmonic function on $W$ upside down. The lack of high
index handles in $W$ implies that this upside down decomposition has
no low index handles (ie. handles with indices below the middle
dimension). This property then implies that the relative homotopy
groups vanish up to the middle dimension. The content of Wall's
previously mentioned theorem is the converse of this simple
observation: if the relative homotopy groups vanish, a handle
decomposition with the required constraint on the indices exists. In
turn, the long exact homotopy sequence for the pair $(W, M)$ then
implies that the vanishing of the relative groups can be formulated by
requiring that the inclusion of $M$ as the boundary induces
isomorphism on homotopy groups up to the required dimension. In short,
we look for conditions for the existence of $W$ for which the
classifying map $\scxs \colon M \to BU$ of the stable almost contact
structure extends, and for which we have a control on the homotopy
groups up to a certain range. The control on the homotopy groups is
provided by the following construction of homotopy theory, called
\emph{Postnikov factorization}: for the map $\scxs \colon M \to BU$
and an integer $k$ there is a fibration $\eta ^k _{\scxs} \colon
B^k_{\scxs}\to BU$ and a map $\bscxs \colon M\to B^k _{\scxs}$ which
factors $\scxs$ into
\[
\scxs =\bscxs\circ \eta ^k _{\scxs}
\]
such that $\bscxs$ induces isomorphisms on the homotopy groups with low index,
while $\eta ^k _{\scxs}$ induces isomorphisms on the homotopy groups
with high index. (For the exact statement see Definition~\ref{def:complex-normal-k-type}.)
}
\dcedit{
\begin{Remark}
The above notion of Postnikov factorization can be found in \cite[Chapter~5.3]{Baues77}, and it is a straighforward
generalization of the construction of a Postnikov tower, which is discussed, for example,
in \cite[p.\,354]{Hatcher}. Indeed, the Postnikov factorizations of the constant map $X\to \{ *\}$ provide the Postnikov tower of $X$.
\end{Remark}}

\dcedit{ Considering now the map $\bscxs\colon M \to B^k_{\scxs}$ (for
  appropriately chosen $k$), the null-bordism of the pair $(M,
  \bscxs)$ provides a manifold $W$ which carries a (stable) complex
  structure, and (through the control on its relevant homotopy groups)
  at the same time admits the required handle decomposition.  We point
  out a further subtlety here: for $f \colon W \to B$ providing the
  null-bordism of $(M, \bscxs)$ we need to require that $f$ also
  covers the classifying map of the (stable) tangent bundle of $W$.
  This ensures that the stable almost complex structure lives on the
  right bundle.  Therefore in defining the bordism groups, we are not
  considering bordism groups of the spaces $B^k_\scxs$ only, but
  rather bordism groups with maps to $BU$ and therefore to $BSO$.

The scheme discussed above can be naturally phrased in terms of $(B, \svb)$-manifolds 
where $\svb \colon B \to BSO$ is a stable vector bundle over $B$.
We briefly recall this setting in Section~\ref{subsec:surgery_settingA}, which allows
us to apply Kreck's approach to surgery theory.}
To be consistent with the existing literature, we will formulate the set-up
using stable \emph{normal} maps, although in our applications we will need
almost complex structures on the tangent bundles of even dimensional manifolds. 
In Section~\ref{subsec:cpx} we discuss the
connection between the stable normal setting and the stable tangential setting and
formulate the basic concepts and definitions of ``complex modified surgery''.
In Section \ref{subsec:stable_and_unstable_surgery} 
we handle the transition from stable complex structures to almost complex structures
and almost contact structures.
After these preliminaries, in 
Section~\ref{subsec:key_surgery_lemmas} we prove our main surgery lemmas 
(the Filling Lemma~\ref{lem:topological_filling}, and 
the Stable and Unstable Surgery Lemmas~\ref{lem:stable_bordism} and 
\ref{lem:unstable_bordism}), which 
lead to the
identification of the obstruction class of Theorem~\ref{thm:main}.
(The contact/symplectic interpretation of this class, and hence the proof of 
Theorem~\ref{thm:main} will be given in Section \ref{sec:complex}.) 
Finally, in Section~\ref{subsec:normal_k_types} we discuss some explicit constructions and
computations in Section \ref{subsec:normal_k_types}.

\subsection{The surgery setting: stable normal bundles} 
\label{subsec:surgery_settingA}
In this subsection we briefly recall the definition of a
``$(B, \svb)$-manifold'' which is a manifold with extra topological structure
on its stable normal bundle.  The theory of $(B, \svb)$-manifolds goes back to
\cite{Lashof63} and was used systematically in the modified surgery
setting of \cite{Kreck99}.  For a detailed treatment of $(B, \svb)$-manifolds
we refer the reader to \cite[Chapter II]{Stong} and \cite[\S 2]{Kreck99}.  

The starting point for our discussion of $(B, \svb)$-manifolds is
a fibration
\[ \svb \colon B \to BSO \]
where $BSO$ is the classifying space of the stable special orthogonal
group $SO$ and $B$ has the homotopy type of a CW complex with a finite number
of cells in each dimension. Since $BSO$ classifies oriented stable vector
bundles, we regard $(B, \svb)$ as an oriented stable vector bundle over $B$.
Given a compact oriented $n$-manifold $X$, let
\[ \nu \colon X \to BSO \]
denote the stable normal Gauss map of $X$.  The stable normal Gauss map $\nu$
is defined by the classifying map of the normal bundle of an embedding
$X \to \R^{n+k}$ for $k>\!\!>n$.  Letting $k$ tend to infinity, the
space of such embeddings is contractible and hence $\nu$ is a
well-defined stable vector bundle over $X$.

A $(B, \svb)$-structure on $X$ is a map 
$\bar \nu \colon X \to B$ which lifts the map $\nu$ over $\svb$: that is, 
there is a commutative diagram:
\[ \xymatrix{  & B \ar[d]^-\svb\\
X \ar[r]^-{\nu} \ar[ur]^-{\bar \nu} & BSO.} \]
A {\em normal $(B, \svb)$-manifold} is a pair $(X, \bar \nu)$ as
above. 

\noindent We now give the basic notions needed in the theory of $(B, \svb)$-manifolds.

\subsection*{Equivalence of $(B, \svb)$-structures:} Two $(B, \svb)$-structures 
$\bar \nu _0$ and $\bar \nu _1$ 
on $X$ are \emph{equivalent}, if there is a $(B, \svb)$-structure $\bar \nu $ on $X \times [0, 1]$ which restricts to the 
$(B, \svb)$-structure $\bar \nu _0$ 
on $X \times \{0\}$ and to $\bar \nu _1$ on
$X \times \{1\}$.
%

%
%
\subsection*{Pullback of $(B, \svb)$-structures:} Given a $(B, \svb)$-structure $\bar \nu_1 \colon
X_1 \to B$ and a diffeomorphism $f \colon X_0 \cong X_1$, there is a
canonical pull-back $(B, \svb)$-structure $f^*(\bar \nu_1)$ on $X_0$ that is given by composing all maps with $f$:
\[ \xymatrix{  & & B \ar[d]^-\svb\\
X_0 \ar[rru]^{f^*(\bar \nu_1)} \ar[r]_(0.45){f} & X_1 \ar[r]_-{\nu_1} \ar[ur]_-{\bar \nu_1} & BSO.} \]
\subsection*{Diffeomorphism of $(B, \svb)$-\dcedit{manifolds}:}  If
$(X_0, \bar \nu _0)$ and $(X_1, \bar \nu _1)$ are $(B, \svb)$-manifolds, a 
{\em $(B, \svb)$-diffeomorphism}
\[ f \colon (X_0, \bar \nu _0) \cong (X_1, \bar \nu _1) \]
is a diffeomorphism $f \colon X_0 \to X_1$ such that $f^*(\bar \nu _1)$
and $\bar \nu _0$ define equivalent $(B, \svb)$-structures on $X_0$.

\subsection*{\dcedit{Changing} orientations of $(B, \svb)$-manifolds:} 
A $(B, \svb)$-structure $\bar \nu$ on $X$ defines a
canonical $(B, \svb)$-structure on $X \times [0, 1]$ via pull-back under the projection to $X$, denoted
$\pi^*(\bar \nu)$.  If $\pi^*(\bar \nu)_i : = \pi^*(\bar \nu)|_{X \times \{ i \}}$, $i = 0, 1$, 
denotes the restriction of $\pi^*(\bar \nu)$ to each end of $X \times [0, 1]$, then 
$\bar \nu = \pi^*(\bar \nu)_1$ and
\[ -\bar \nu : = (\pi^*\bar \nu)_0 = \pi^*(\bar \nu|_{X \times \{ 0 \}})  \]
is the $(B, \svb)$-structure defined on the other end of $X \times [0, 1]$ via 
$\pi^*(\bar \nu)$.  

\dcedit{
\subsection*{Surgery on $(B, \svb)$-manifolds}
Let $(X, \bar \nu)$ be a compact $n$-dimensional $(B, \svb)$-manifold, possibly with boundary, 
and let $h_{k{+}1} := D^{k{+}1} \times D^{n-k}$ denote an $(n+1)$-dimensional $(k{+}1)$-handle.
For a {\em $(B, \svb)$ $k$-surgery} on $(X, \bar \nu)$ we require the following data:
\begin{enumerate}
\item An embedding $\phi \colon S^k \times D^{n-k} \to \textup{int}(X)$ to the interior of $X$;
\item A $(B, \svb)$-structure $\bar \nu_{W_\phi}$ on the trace of surgery on $\phi$,
\[ W_\phi : = (X \times I) \cup_\phi h_{k{+}1},\]
which extends the natural $(B, \svb)$-structure 
$\bar \nu \times I$ on $X \times I \subset X_\phi$ induced by $\bar \nu$.
\end{enumerate}
The boundary of $W_\phi$ is the union $X \cup_{\del X} X_\phi$ where $X_\phi$
is the result of surgery on $\phi$.  The restriction of the $(B, \svb)$-structure
$\bar \nu_{W_\phi}$ to $X_\phi$ is a $(B, \svb)$-structure $\bar \nu_\phi$ on $X_\phi$,
giving the $(B, \svb)$-manifold $(X_\phi, \bar \nu_{\phi})$.
}

\dcedit{
\subsection*{Bordism of $(B, \svb)$-manifolds:} Suppose that
 $W$ is an $(n+1)$-manifold with boundary 
\jbedit{
$\del W = X_0 \cup_YX_1$, that is, $\partial W$ is the union of two compact $n$-manifolds $X_0$
  and $X_1$, both with boundary $Y$}. If $\bar \nu_W \colon W \to
B$ is a $(B, \svb)$-structure on $W$, then $\bar \nu_W$ restricts to
give \jbedit{$(B, \svb)$}-structures $\bar \nu_0 \colon X_0 \to B$ and
$\bar \nu _1 \colon X_1 \to B$.  In this case $(-X_0, -\bar \nu_0)$
and $(X_1, \bar \nu_1)$ are called $(B, \svb)$-bordant rel.~boundary.
If $Y = \phi$ is empty, this gives the defintion of $(B,
\svb)$-bordism of closed $(B, \svb)$-manifolds.  } \dcedit{
\begin{Example}[$(B, \svb)$-diffeomorphism implies $(B, \svb)$-bordism] \label{ex:Diffeomorphis_and_bordism}
If $f \colon (X_0, \nu_0) \cong (X_1, \nu_1)$ is a
$(B, \svb)$-diffeomorphism between closed manifolds then the $s$-cobordism
\[ (X_0 \times [0, 1]) \cup_f (X_1 \times [1, 2]) \]
admits the structure of a $(B, \svb)$-bordism between $(X_0, \bar \nu _0)$ and
$(X_1, \bar \nu _1)$.  In particular, closed $(B, \svb)$-diffeomorphic manifolds are
$(B, \svb)$-bordant.
\end{Example}
}
\dcedit{
\begin{Example}[$(B, \svb)$-bordism and $(B, \svb)$-surgery] \label{ex:Bordism_and_surgery}
By \cite[Theorem 1]{Milnor61} extended to the case of compact manifolds,
every $(B, \svb)$-bordism $(W, \bar \nu_W; X_0, X_1)$ is $(B, \svb)$-diffeomorphic to the trace of a finite sequence of $(B, \svb)$-surgeries on $(X_0, \bar \nu_0)$.
\end{Example}

The set of $(B, \svb)$-bordism classes of closed $n$-dimensional $(B, \svb)$-manifolds $(N, \bar \nu)$ form an
abelian group
\[ \Omega_n(B; \svb) := \{ [N, \bar \nu] \} \]
with addition  given by disjoint union and inverse given by $-[M, \bar \nu] = [-M, -\bar \nu]$.

\medskip
}
We give some examples of stable bundles $(B, \svb)$ which we shall use later.
\dcedit{
\begin{Example} \label{ex:Complex_bordism}
Consider $(B, \svb) = (BU, F)$ where, if $\pi_{SO} \colon ESO \to BSO$
is a model for the universal principal $SO$-bundle, we take $BU$ to be the space
$(ESO)/U$, with $U$ acting on $ESO$ via the inclusion $U
\hookrightarrow SO$, and $F$ to be the map induced by $\pi_{SO}$.  The
map $F \colon BU \to BO$ is a bundle map with fibre $SO/U$ and
corresponds to forgetting almost complex structures on the level of
classifying spaces.  $(BU, F)$-bordism is then just complex bordism so
that $\Omega_n(BU; F) = \Omega^U_n$.  In particular by \cite[Theorem
  p.117]{Stong}, we have that $\Omega_{2q{+}1}^U = 0$ and every odd-dimensional
stably complex manifold bounds a stably complex manifold.
\end{Example}
}
\dcedit{
\begin{Example} \label{ex:Connective_complex_bordism}
$(B, \svb) = (BU\an{k{+}1}, \pi_{k{+}1})$, where 
$\pi_{k{+}1} \colon BU\an{k{+}1} \to BU$ is the $k^{th}$ connective cover of $BU$: that is, the universal map such that
$\pi_i(BU\an{q{+}1}) = 0$ for $i \leq q$ and 
$(\pi_{k{+}1})_* \colon \pi_i(BU\an{k{+}1}) \cong \pi_i(BU)$ for $i \geq k{+}1$.
For example $(BU\an{4}, \pi_4) = (BSU, \pi_{SU})$, where $SU$ denotes the \jbedit{stable special unitary group.}  In this case $\Omega_*(BSU; \pi_{SU}) = \Omega^{SU}_*$ is
special unitary bordism: see \cite[Chapter X]{Stong}.
By \cite[p.\,248]{Stong}, $\Omega_9^{SU} {\cong \Z_2}$ and in Lemma \ref{lem:alpha_spheres} we see that
the non-zero element in this group is represented by an exotic sphere $\Sigma$ which does not
admit a Stein filling.

In general the groups $\Omega_*(BU\an{k{+}1}; \pi_{k{+}1})$ become harder to compute
as $k$ becomes larger: we investigate the groups $\Omega_{2k{+}1}(BU\an{k}; \pi_k)$ and 
$\Omega_{2k{+}1}(BU\an{k{+}1}; \pi_{k{+}1})$ further in \cite{BCS}.
\end{Example}
}

\dcedit{
\subsection*{Surgery below the middle dimension}
For our purposes, the utility of the additional normal structure on $(B, \svb)$-manifolds 
stems from the following result of Kreck, which we re-state with minor modifications:
note that $[\frac{n}{2}]$ denotes the greatest integer less than or equal to $n/2$.

\begin{Proposition}[{\cite[Proposition 4]{Kreck99}}]\label{prop:Kreck_surgery}
Let $\mu \colon B \to BSO$ be a fibration and assume that $B$ is connected and has the homotopy type of a $CW$-complex with finite $[\frac{n}{2}]$-skeleton. Let $\bar \nu$ be a normal \jbedit{$(B, \svb)$}-structure on an $n$-dimensional compact manifold $X$. Then, if $n \ge 4$, by a finite sequence of $(B, \svb)$-surgeries $(X, \bar \nu)$ can be replaced by $(X',\bar \nu')$ so that $\bar \nu'$ is an $[\frac{n}{2}]$-equivalence.  In particular, $(X',\bar \nu')$ is again a normal $(B, \svb)$-manifold, i.e. the following diagram commutes:
\[ \xymatrix{  & B \ar[d]^-\svb\\
X' \ar[r]^-{\nu'} \ar[ur]^-{\bar \nu'} & BSO.} \]
\end{Proposition}

\begin{Remark} \label{rem:Kreck_error}
We point out that the proof of \cite[Proposition 4]{Kreck99} contains
an error.  It uses \cite[Lemma 3]{Kreck99} which relies on the claim
that the higher homotopy groups $\pi_i(B)$ are finitely generated over
the group ring $\Z\pi$, where $\pi = \pi_1(B)$.  This is not true in
general: see e.g.~\cite[p.\,423]{Hatcher}.  In the proof of
\cite[Proposition 3.1]{BCSS2}, which is a modified and specialised
version of \cite[Proposition 4]{Kreck99}, we make the same error.
However, as we explain below, it is a simple matter to correct the
proof using standard methods for surgery below the middle dimension as
in \cite[Lemma 3.55]{Lueck}; see also \cite[Theorem 1.2]{Wall99}.
\end{Remark}
}

\dcedit{
\begin{proof}[Outline of the proof of Proposition \ref{prop:Kreck_surgery}]
We work inductively making the maps $\bar \nu \colon X \to B$ more and more connected.
By performing zero surgeries we may assume that $X$ is connected and so $\bar \nu$ is 
an isomorphism on $\pi_0$.  Since $\pi_1(B)$ is finitely generated, further zero-surgeries
allow us to make $\bar \nu$ surjective on $\pi_1$.
We then introduce a finite number of relations by doing surgery along a finite collection of 
embedded curves ($n \ge 4$) to obtain an isomorphism of fundamental groups: see \cite[p.\,718]{Kreck99}. 
Hence, we may assume for $1 \leq i < [\frac{n}{2}]$, that $\bar \nu$ is an isomorphism on $\pi_{j}$
for $j \leq i$.
From the long exact homotopy sequence,
\[ \dots \to \pi_j(X) \xra{\bar \nu_*} \pi_j(B) \to \pi_j(B, X) \to \pi_{j-1}(X) \to \dots, \]
where we regard $\bar \nu$ as an inclusion, it follows that $\pi_j(B, X) = 0$ for $j \leq i$.
The assumption that $B$ has the homotopy type of a 
$CW$-complex with finite $[\frac{n}{2}]$-skeleton means that we can apply
\cite[Lemma 3.55]{Lueck} which states that $\pi_{i+1}(B, X)$ 
is finitely generated over $\Z\pi$: Here is a short sketch of the proof.
Since the pair $(B, X)$ is $i$-connected, the relative Hurewicz Theorem states that 
$\pi_{i+1} (B, X) \cong H_{i+1} (B, X; \Z\pi)$.  Since 
$H_{i+1}(B, X; \Z\pi)$ is the first non-trivial homology group of a chain complex 
of finitely generated free $\Z\pi$-modules, it is finitely generated.

We choose $\{x_1, \dots, x_k\}$, a set of generators for $\pi_{i+1}(B, X)$.  The class $x_1$
is represented by a commutative diagram
\[ \xymatrix{ S^i \ar[d] \ar[r]^{\del \wh x_1} & X \ar[d]^{\bar \nu} \\
D^{i+1} \ar[r]^{\wh x_1} & B. } \]
For dimensional reasons we can assume that 
$\del \wh x_1 \colon S^i \to X$ is an embedding.
For any $i$ the stable tangent bundle $\tau_{S^i}$ 
is stably trivial and $\tau_X|_{\del \wh x_1(S^i)}$ is pulled back
from $B$ along a homotopically trivial map, hence
$\tau_X|_{\del \wh x_1(S^i)} = \nu_{\del \wh x_1} \oplus \tau_{S^i}$ implies that  the normal bundle, 
$\nu_{\del \wh x_1}$, of $\del \wh x_1(S^i)$ in $X$ is stably trivial. 
Since the rank of $\nu_{\del \wh x_1}$ is greater than $i$,
it follows that $\nu_{\del \wh x_1}$ is trivial and so $\del \wh x_1$ extends
to an embedding $\phi_1 \colon D^{n-i} \times S^i \to X$.  Moreover, the map
$\wh x_1$ gives the extra data to perform a $(B, \svb)$-surgery on $\phi_1$:
see \cite[Theorem 3.59]{Lueck}.
Let $W = (X \times I) \cup h_{i+1}$ be the trace of such a surgery, with normal map
$\bar \nu_W \colon W \to B$.  The relative group $\pi_{i+1}(W, X)$
is a rank one free $\Z\pi$-module whose generator is mapped to $x_1$ under the 
map $\pi_{i+1}(W, X) \to \pi_{i+1}(B, W)$ induced by the inclusion
$i_X \colon X \to W$.  The long exact sequence
\[ \dots \to \pi_{i+1}(W, X) \to \pi_{i+1}(B, X) \to \pi_{i+1}(B, W) \to \pi_i(X, W) \to \dots
 \]
shows that $\pi_{i+1}(B, W)$ is generated by $\{(i_X)_*(x_2), \dots, (i_X)_*(x_k)\}$.
If $(X', \bar \nu') \subset (W, \bar \nu_W)$ denotes the outcome of the $(B, \svb)$-surgery
on $\phi_1$, then $W \cong (X' \times I) \cup h^{n-i}$ and since 
$i < [\frac{n}{2}]$, $n-i > [\frac{n}{2}]$.  From the long exact sequence
\[ \dots \to \pi_{j}(W, X') \to \pi_{j}(B, X') \to \pi_{j}(B, W) \to \pi_{j-1}(X', W) \to \dots~,\]
we see that $\pi_{j}(B, X') \cong \pi_{j}(B, W)$ for $j \leq [\frac{n}{2}]$ 
and so, by a finite sequence of $(B, \svb)$-surgeries
we can achieve that the relative groups $\pi_{j}(B, X')$ vanish for $j \leq i{+}1$.
Proceeding by induction on $i$, by a finite sequence of $(B, \svb)$-surgeries
we can achieve that the relative groups $\pi_{j}(B, X')$ vanish for $j \leq [\frac{n}{2}]$; i.e.~
$\bar \nu'$ is a $[\frac{n}{2}]$-equivalence.
\end{proof}
}

\subsection{Stable complex structures} \label{subsec:cpx}
An example of $(B, \svb)$-manifolds of particular interest in this paper
(recall Example \ref{ex:Complex_bordism}) is given by
\[ (B, \svb) = (BU, F) \]
where $F \colon BU \to BO$ is the canonical forgetful map between
classifying spaces.  A $(BU, F)$-manifold is nothing but a stably
complex manifold.  Notice that an almost complex structure $\acxs$ on a
$2q$-manifold $X$ (that is, a reduction of the structure group of the 
tangent bundle of $X$ from $SO(2q)$ to
$U(q)$) naturally induces a stable complex structure on $\tau_X$, 
the stable tangent bundle of $X$.
As there is a canonical bundle isomorphism,
\[ \tau_X \oplus \nu_X \cong \varepsilon, \]
where $\nu_X$ is the stable normal bundle of $X$ and $\varepsilon$
denotes the trivial stable bundle, a stable complex structure
on $\tau_X$ induces a stable complex structure on $\nu_X$:
choose the unique stable complex structure on $\nu_X$ so that the sum
with the given stable complex structure on $\tau_X$ is the trivial
stable complex structure on $\varepsilon.$
We shall denote the stable normal complex structure associated to $(X, \acxs)$ by 
$S \acxs$ or sometimes $\scxs_X$.

As in the even-dimensional case, an almost contact
structure $\acs$ on a closed $(2q{+}1)$-manifold $M$ (that is, the reduction of the
structure group from $SO(2q{+}1)$ to $U(q)$) induces a stable
complex structure $S\acs = \scxs _M$ on the stable normal bundle of $M$. 
(We will also call the stabilized structures \emph{complex} rather than 
\emph{contact} in the odd-dimensional case.)  Since stable tangential complex and stable normal complex
structures determine each other, we will focus on the normal picture
(although in the applications we will need results for the tangential
structures).

Building on the discussion of $(B, \svb)$-manifolds from Section 
\ref{subsec:surgery_settingA}, we now establish the basic notions in 
stable complex surgery which we shall use throughout this paper.
Let
\[ \eta \colon B \to BU \]
be a fibration, where, as before, $B$ has the homotopy type of a CW complex with a
finite number of cells in each dimension.
We regard $(B, \eta)$ as a stable complex vector bundle over $B$ with
underlying oriented bundle $F \circ \eta \colon B \to BSO$.
We shall be interested in the situation described by the following
commutative diagram:
\begin{equation} \label{eq:Beta}
	\xymatrix{  & B \ar[d]^-\eta \ar[dr]^{F \circ \eta}\\
X \ar[r]^-{\scxs} \ar[ur]^-{\bscxs} & BU \ar[r]^(0.4)F & BSO.} 
\end{equation}
Here $X$ is an oriented manifold with stable normal bundle $\nu = F \circ \scxs$,
$(X, \scxs)$ is a compatibly oriented stably complex manifold and $(X, \bscxs)$ is
a $(B, F \circ \eta)$-manifold.  Since $F$ is fixed, we shall usually call $(X, \bscxs)$
a $(B, \eta)$-manifold for short: this simply means that $(X, \bscxs)$ is a 
$(B, F \circ \eta)$-manifold.
\begin{Definition}[$\scxs$-compatible $(B, \eta)$-manifold] \label{def:zeta-compatible}
In the situation of the commutative diagram \eqref{eq:Beta} above,
we say that $\bscxs \colon X \to B$ is a $\scxs$-compatible $(B, \eta)$-manifold;
i.e., $(X, \bscxs)$ is a $(B, F \circ \eta)$-manifold with underlying
stably complex manifold $(X, \scxs)$.
\end{Definition}

It follows from the definitions that $(-X, -\bscxs)$ is a 
$(-\scxs)$-compatible $(B, \eta)$-manifold and that a $(B, F \circ \eta)$-diffeomorphism
$ f \colon (X_0, \bscxs_0) \cong (X_1, \bscxs_1) $
is also a stably complex diffeomorphism $f \colon (X_0, \eta \circ \scxs_0) \cong (X_1,
\eta \circ \scxs_1)$.  
$(B, \eta)$-bordism groups are defined as follows.

\begin{Definition}[$(B, \eta)$-bordism] \label{def:Beta-bordism}
We define
\[ \Omega_n(B; \eta) := \Omega_n(B; F \circ \eta) \]
to be the bordism group of $(B, F \circ \eta)$-bordism classes of 
closed $n$-dimensional $(B, F \circ \eta)$-manifolds as defined
in Section \ref{subsec:surgery_settingA}.
\end{Definition}
The following definition (as explained at the beginning of this section) is of
fundamental importance in our present discussion.
\begin{Definition}[Normal $k$-smoothing] \label{def:normal-k-smoothing}
A normal $k$-smoothing in $(B, \eta)$ is a normal $(B,
\eta)$-manifold $(X, \bscxs)$, where $\bscxs \colon X \to B$ is a
$(k{+}1)$-equivalence; i.e.~$\bscxs$ induces an isomorphism on
homotopy groups $\pi_i$ for $i \leq k$ and a surjection on
$\pi_{k{+}1}$.
\end{Definition}

For the purposes of understanding Stein fillings, we shall be
interested in the case where $\bscxs$ is a $k$-smoothing for certain
$k$.
One may ask whether for some stably complex manifold $(X, \scxs)$ 
there are any $\scxs$-compatible normal $k$-smoothings $\bscxs \colon X \to B$ at all.  
In fact this is always the
case because the map $\scxs \colon X \to BU$ can be factorised (up to homotopy) as a
composition
\[ X \xra{~\bscxs~} B^k_\scxs \xra{~\eta^k_\scxs~} BU,    \]
where $\bscxs$ is a $(k{+}1)$-equivalence and $\eta^k_\scxs$ is a fibration.  
The space $B^k_\scxs$ and the maps $\bscxs$ and $\eta^k_\scxs$ make up the 
$k^{th}$ Postnikov factorisation of $\scxs$.  The existence of the $k^{th}$ Postnikov
factoriation is proven in \cite[Theorem 5.3.1]{Baues77} \dcedit{and \cite[Proposition 4.13]{Hatcher}}, 
its defining properties are identified in 
Definition \ref{def:complex-normal-k-type} below and we discuss some examples
in Section \ref{subsec:normal_k_types}.  In general, for any $k \geq 0$, a map $f \colon X \to Y$
between $CW$-complexes, has a $k^{th}$ Postnikov factorisation $f \simeq \eta^k_f \circ \bar f$
by maps $\bar f \colon X \to Y^k_f$ and $\eta^k \colon Y^k_f \to Y$.
Such factorisations 
are built by first converting $f$ into a fibration and then working
inductively so that there are fibrations $Y^k_f \to Y^{k-1}_f$ with
fibre $K(\pi_k(F), k)$ where $F$ is the homotopy fibre of $f \colon X \to Y$.

\begin{Definition}[Complex normal $k$-type] \label{def:complex-normal-k-type}
Let $(X, \scxs)$ be a stably complex manifold.  The complex normal
$k$-type of $(X, \scxs)$, denoted $(B^{k}_{\scxs}, \eta^k_\scxs)$, is
defined to be the fibre homotopy type of the fibration $\eta^k_\scxs$
in the following diagram:
\[ \xymatrix{  & B^{k}_\scxs \ar[d]^-{\eta^k_\scxs}\\
X \ar[r]^-{\scxs} \ar[ur]^(0.55){\bscxs} & BU.} \]
The fibration $\eta^k_\scxs$ is uniquely defined up to fibre homotopy
type by the following properties:
\begin{enumerate}
\item
the map $\bscxs$ is a $(k{+}1)$-equivalence,
\item
the map $\eta^k_\scxs$ is a $(k{+}1)$-coequivalence,
i.e.\,\,$(\eta^k_{\scxs})_* \colon \pi_j(B^{k}_\scxs) \to \pi_j(BU)$
is injective when $j = k{+}1$ and an isomorphism if $j > k{+}1$.
\end{enumerate}
\end{Definition}

We conclude this subsection by considering the role of the choice of
the normal $(q-1)$-smoothing $\bscxs \colon X \to B$ on the
bordism class $[X, \bscxs] \in \Omega_{2q{+}1}(B: \eta)$.  Our method is
to adapt the key point of the proof of \cite[Proposition 7.4]{Kreck85}
to the complex setting.
Given a stable complex vector bundle $\eta \colon B \to BU$, let
$\Aut(B, \eta)$ be the group of fibre homotopy classes of fibre
self-homotopy equivalences of $\eta$.  That is, $\Aut(B, \eta)$
consists of fibre homotopy equivalence classes of maps $\alpha \colon
B \simeq B$ which make the following diagram commute:
\[  \xymatrix{ B \ar[r]^\alpha \ar[d]^\eta & B \ar[d]^\eta\\
BU \ar[r]^{\rm Id} & BU.  } \]
The group $\Aut(B, \eta)$ acts on the set of $(B,
\eta)$-diffeomorphism classes of complex normal $k$-smoothings in $(B,
\eta)$ by mapping a complex normal $k$-smoothing $\bscxs \colon X \to
B$ to the complex normal $k$-smoothing $\alpha \circ \bscxs \colon X
\to B$.

\begin{Lemma}[cf.~{\cite[Proposition 7.4]{Kreck85}}]  \label{lem:Aut(B^k)}
Suppose that $(X_0, \scxs_0)$ and $(X_1, \scxs_1)$ are stably complex 
manifolds and that for $i = 0, 1$, $\bscxs_i \colon X_i \to B^k_{\scxs_0}$ is a
$\scxs_i$-compatible normal $k$-smoothing in $(B^k_{\scxs_0}, \eta^k_{\scxs_0})$,
the complex normal $k$-type of $(X_0, \scxs_0)$.  If 
$f \colon (X_0, \scxs_0) \cong (X_1, \scxs_1)$ is a stably complex
diffeomorphism, then there is a fibre homotopy self-equivalence $\alpha \in
\Aut(B^k_{\scxs_0}, \eta^k_{\scxs_0})$ such that $f$ is a
$(B^k_{\scxs_0}, \eta^k_{\scxs_0})$-diffeomorphism from $(X_0, \alpha
\circ \bscxs_0)$ to $(X_1, \bscxs_1)$.
\end{Lemma}

\begin{proof}
	The maps
	\[ \bscxs_0, ~ f^*(\bscxs_1) \colon X_0 \to B^k_{\scxs_0} \]
	determine two complex normal $k$-smoothings on $X_0$.  Now the
        universal properties of Postnikov stages of maps
        \cite[Corollary 5.3.8]{Baues77} ensure that there is a fibre
        homotopy equivalence $\alpha \colon (B^k_{\scxs_0},
        \eta^k_{\scxs_0}) \simeq (B^k_{\scxs_0}, \eta^k_{\scxs_0})$
        such that $\alpha \circ \bscxs_0$ and $f^*(\bscxs_1)$
        are equivalent $(B^k_{\scxs_0}, \eta^k_{\scxs_0})$-structures
        on $X_0$.  By definition, this means that $f$ is a
        $(B^k_{\scxs_0}, \eta^k_{\scxs_0})$-diffeomorphism from $(X_0,
        \alpha \circ \bscxs_0)$ to $(X_1, \bscxs_1)$.
\end{proof}

The following corollary is an important consequence of Lemma
\ref{lem:Aut(B^k)} which arises from the fact that the induced action
of $\Aut(B^k_{\scxs}, \eta^k_\scxs)$ on $\Omega_n(B^k_\scxs;
\eta^k_\scxs)$ is by group automorphisms.

\begin{Corollary}  \label{cor:Postnikov}
Let $(B^k_\scxs, \eta^k_\scxs)$ be the normal $k$-type of $(X,
\scxs)$.  If $\bscxs \colon X \to B^k_\scxs$ is a closed
$\scxs$-compatible normal $k$-smoothing such that $[X, \bscxs]= 0 \in
\Omega_n(B^k_\scxs; \eta^k_\scxs)$, then for all $\scxs$-compatible
normal $k$-smoothings $\hat \scxs \colon X \to B^k_\scxs$ we have that
$[X, \hat \scxs] = 0 \in \Omega_n(B^k_\scxs; \eta^k_\scxs)$.
\end{Corollary}

\begin{proof}
	Applying Lemma \ref{lem:Aut(B^k)} to $(X, \bscxs)$ and $(X,
        \hat \scxs)$ we deduce that there is a fibre homotopy equivalence
        $\alpha \colon (B^k_\scxs, \eta^k_\scxs) \simeq (B^k_\scxs,
        \eta^k_\scxs)$ such that $(X, \hat \scxs)$ is $(B^k_\scxs,
        \eta^k_\scxs)$-diffeomorphic to $(X, \alpha \circ \bscxs)$.
        Hence $[X, \hat \scxs] = \alpha_*([X, \bscxs]) = 0$.
\end{proof}


\subsection{Stable and unstable surgery} 
\label{subsec:stable_and_unstable_surgery}
Propagating contact structures over Weinstein handles
requires information about almost contact structures of $M$.
On the other hand, computations in $(B, \eta)$-bordism require
the use of the stable normal bundle of $M$.  In this subsection we
prove two lemmas which allow us to move between these two settings.

Recall that $h_{k{+}1} := D^{k{+}1} \times D^{n-k}$ is an $(n+1)$-dimensional $(k{+}1)$-handle.
Let $(M, \acs)$ be a closed $(2q{+}1)$-dimensional almost contact manifold
and set $n := 2q{+}1$.  For an {\em almost complex $k$-surgery} on 
$(M, \acs)$ we require the following data:
\begin{enumerate}
\item An embedding $\phi \colon S^k \times D^{n-k} \to M$;
\item An almost complex structure $\acxs$ on 
\[ W_\phi : = (M \times I) \cup_\phi h_{k{+}1},\]
extending the natural almost complex structure $\varphi \times I$ on $M \times I \subset W_\phi$ induced by $\varphi$.
\end{enumerate}
The result of this surgery is the other boundary component of $W_\phi$, denoted
$M_\phi$.  It is an almost contact manifold with almost contact structure
$\acs_\phi := \acxs|_{M_\phi}$.

\begin{Definition}[Almost complex surgery] \label{def:almost_complex_surgery}
In the situation above, we shall say that the almost contact manifold
$(M_\phi, \acs_\phi)$ is obtained from $(M, \acs)$ by a $k$-dimensional almost complex surgery.
\end{Definition}

When we work with stably complex manifolds, we have the analogous
situation, where almost contact structures and almost complex
structures on the tangent bundle are replaced first by stable complex
structures on the tangent bundle and then by stable complex structures
on the normal bundle.  Thus to perform stable complex $k$-surgery on a
stably complex $n$-manifold $(M, \scxs)$, we require an embedding $\phi
\colon S^k \times D^{n-k} \to M$ along with an extension of the stable
complex structure $\scxs$ to a stable complex structure on the trace
of the surgery
$ W_\phi = (M \times I) \cup_\phi h_{k{+}1}.$
\dcedit{
In this case $M_\phi$ inherits the stable complex structure $\scxs_\phi : = \scxs|_{M_\phi}$,
and we shall say that $(M_\phi, \scxs_\phi)$ is obtained from $(M, \scxs)$ via stable complex surgery.}

\dcedit{
\begin{Notation} \label{not:acs_and_scxs}
Given an almost contact manifold $(M, \acs)$, let  $(M, \scxs)$
denote the stably complex manifold defined by $\acs$.
\end{Notation}
}

\begin{Lemma} \label{lem:destabilising}
Let $(M, \acs)$ be a $(2q{+}1)$-dimensional almost contact
manifold and suppose $(W, \scxs_W; M, M_\phi)$ is the trace of a stable
complex $k$-surgery on $(M, \scxs)$ with $k \leq 2q$.
Then there is an almost complex structure
$\acxs$ on $W$ with $S \acxs = \scxs_W$ and which restricts to
$\acs \times I$ on $M \times I$.
\end{Lemma}
(Notice that the above lemma is stated for $k\leq 2q$ --- in our
applications, however, we will use the statement only in the range
$k\leq q{+}1$.)

\begin{proof}
	First we notice that the stable normal complex structure on
        $W$ can be converted to a stable tangential complex structure
        on $W$ which stabilies $\acs \times I$ when restricted to $M
        \times I$.  Our problem is to reduce the structure
        group of the tangent bundle
        of $W$ to $U(q{+}1)$, and we must do this relative to the chosen
        reduction corresponding to $\acs \times I$ on $M \times I$.
        We encounter the unstable lifting problem which maps to the
        stable lifting problem:
	\[ \xymatrix{ M \times I \ar[d] \ar[r]^{\acs \times I} & BU(q) \ar[r] & BU(q{+}1) \ar[d] \ar[r] & BU \ar[d] \\
	W \ar@{.>}[urr] \ar[rr]^(0.45){\tau_W} & & BSO(2q{+}2) \ar[r] & BSO,   }  \]
	where $\tau_W$ classifies the tangent bundle of $W$.	
	The lifting obstructions for these problems lie in the groups
	\[ H^{k{+}1}(W, M; \pi_{k}(SO(2q{+}2)/U(q{+}1))) \quad \text{and} \quad 
	H^{k{+}1}(W, M; \pi_{k}(SO/U)), \]
        and the unstable lifting obstruction maps to the stable
        lifting obstruction under the coefficient homomorphism $S_*
        \colon \pi_k(SO(2q{+}2)/U(q{+}1)) \to \pi_k(SO/U)$.  The
        map $S_*$ is an isomorphism for $k \leq 2q$ by
        \cite[p.\,432]{Gray59}.  Hence the vanishing of the stable
        obstruction ensures the vanishing of the unstable obstruction.
        It follows therefore that there is an almost complex structure
$\acxs$ on $W$ compatible with $(M \times I, \acs \times I)$.
\end{proof}
	Note that in the setting of Lemma \ref{lem:destabilising} there may
        be several homotopy classes of almost contact structures $\acs'$ on
        $M_\phi$ such that $S\acs'$ is homotopic to $\scxs_\phi$.  The almost
        complex structure $\acxs$ will induce one such structure
        $\acxs|_{M_\phi}$ on $M_\phi$.  To obtain an almost complex
        bordism $(W, \acxs; (M, \acs), (M_\phi, \acs'))$ for a
        specific almost contact structure $\acs'$ we may need to
find an alternative bordism. 

\dcedit{		
\begin{Lemma} \label{lem:stable_to_unstable}
Let $(M, \scxs)$ be a closed connected stably complex $(2q{+}1)$-manifold.
\begin{enumerate}
\item \label{lem:stable_to_unstable:onto}
There is an almost contact structure $\acs$ such that $\scxs = S \acs$.
\item \label{lem:stable_to_unstable:kernel}
Suppose that two almost contact structures $\acs$ and
$\acs '$ are such that $S \acs = S\acs' = \scxs$.
\jbedit{Then there} is an almost contact structure $\acs _d$ on the sphere $S^{2q{+}1}$ 
such that:
\begin{enumerate}
\item $(M,\acs )\# (S^{2q{+}1} , \acs _d)$ and $(M, \acs ')$ are equivalent 
as almost contact manifolds,
\item $S\acs_d = S\acs_0$ where $\acs_0$
is the standard almost contact structure on $S^{2q{+}1}$,
\item $(S^{2q{+}1}, \acs _d)$ bounds an almost complex $(2q{+}2)$-manifold with a handle 
decomposition of handles of indices $\leq q{+}1$.
\end{enumerate}
\end{enumerate}
\end{Lemma}
}

\begin{proof}
\dcedit{
Following Gray \cite[p.\,432]{Gray59}, we write
\[ F_{2q{+}1} = SO(2q{+}1)/U(q) \quad \text{and} \quad F_{2q{+}2} = SO(2q{+}2)/U(q{+}1). \]
Gray \cite[p.\,432]{Gray59} shows that there is a homeomorphism $F_{2q{+}1} \cong F_{2q{+}2}$ 
and a fibre bundle
\begin{equation} \label{eq:Gray}
F_{2q{+}1} \to F_{2q+3} \to S^{2q{+}2}. 
\end{equation}
As Sato \cite[Proposition 1]{Sato} observes, the homeomorphism $F_{2q{+}1} \cong F_{2q{+}2}$
entails that there is a bijection between the set of homotopy classes of almost contact structures
on $M$ and the set of homotopy classes of almost complex structures on 
$M \times I$.  Since there is also a bijection between the sets of homotopy
classes of stable complex structures on $M$ and on $M \times I$, it
suffices to prove that every stable complex structure on $M \times I$
is the stabilisation of an almost complex structure on $M \times I$.
But this last point is true for dimensional reasons, as we now explain.

Let $V$ be any oriented vector bundle of real rank $2j$ over a space
$X$ which is homotopy equivalent to a closed connected oriented
$(2j{-}1)$-manifold.  Let $F(V) \to X$ and $F^S(V) \to X$ be,
respectively the \jbedit{bundle of oriented, orthogonal frames and its
  stable analogue}: $F(V)$ is a principal $SO(2j)$-bundle over $X$ and
$F^S(V)$ is a principal $SO$ bundle over $X$.  In addition, we define
$\cal{J}(V) := F(V)/U(j)$ and $\cal{J}^S(V) := F^S(V)/U$ and the
stabilisation map
\[ S \colon \cal{J}(V) \to \cal{J}^S(V) .\]
A complex structure on $V$ corresponds to a section
of the $SO(2j)/U(j)$-bundle $\cal{J}(V) \to X$, a homotopy of complex
structures corresponds to a homotopy of sections and a similar statement
holds for stable complex structures on $V$ and sections of the $SO/U$-bundle 
$\cal{J}^S(V) \to X$.  
To compare stable and unstable complex structures on $V$ we consider the following diagram:
\[ \xymatrix{ & \cal{J}(V) \ar[d] \ar[r]^(0.45)S & \cal{J}^S(V) \ar[d] \\
X \ar@{.>}[ur] \ar[r]^{\Id}& X \ar[r]^{\Id} & X  } \]
Note that the homeomorphism $F_{2j{-}1} \cong F_{2j}$ and the fibre
bundle \eqref{eq:Gray} show that the stabilisation map $S \colon
F_{2j} \to SO/U$ is a $(2j{-}1)$-equivalence. Let $X^{(2j{-}2)}
\subset X$ be a $(2j{-}2)$-skeleton.  Elementary obstruction theory
applied to the lifing problem above shows the following:
\begin{enumerate}
\item Every section $\jbedit{\sigma}^S$ of $\cal{J}^S(V) \to X$ is homotopic to \jbedit{some stabilisation} $S \circ \jbedit{\sigma}$ where $\jbedit{\sigma}$ is a section 
of $\cal{J}(V) \to X$.
\item If $\jbedit{\sigma}$ and $\jbedit{\sigma}'$ are sections of $\cal{J}(V) \to X$ such that $S \circ \jbedit{\sigma}$ and $S \circ \jbedit{\sigma}'$
are homotopic, then there is a homotopy $H$ between the restrictions $\jbedit{\sigma}|_{X^{(2j{-}2)}}$ and $\jbedit{\sigma}'|_{X^{(2j{-}2)}}$.
Moreover, the homotopy $H$ defines an obstruction class,
\[ \mathfrak{o}_H \in H^{2j}(X; A_{2j{-}1}), \]
where $A_{2j{-}1} := \ker(\pi_{2j{-}1}(SO(2j)/U(j)) \to \pi_{2j{-}1}(SO/U))$, and if $\mathfrak{o}_H = 0$, then
$\jbedit{\sigma}$ and $\jbedit{\sigma}'$ are homotopic.
\end{enumerate}
This proves \eqref{lem:stable_to_unstable:onto} and parts (a) and (b) of \eqref{lem:stable_to_unstable:kernel}.

We now turn to the proof of \eqref{lem:stable_to_unstable:kernel}~(c).
The homotopy classes of almost contact structures on $S^{2q{+}1}$
which are stably equivalent to the standard almost contact structure
$\acs_0$ are parametrised by the group $A_{2q{+}1}$ define above.  As
noted in \cite[p.\,1201]{Geiges97}, results of Sato \cite{Sato},
building on work of Morita, show that for $q$ even one can realise all
these homotopy classes $\acs$ via contact structures on various
standard Brieskorn spheres and that all of these $\acs$ are Stein
fillable.}  The Stein fillings then provide the required almost
complex $(2q{+}2)$-manifolds.

A similar observation holds when $q$ is odd, by utilising the
calculations of \cite{Ding&Geiges}. One simply takes the product of
even dimensional spheres $W = S^{q{+}1} \times S^{q{+}1}$. The stable tangent bundle of $W$ is trivial, 
so we choose a stable
trivialisation. After removing a ball we obtain a stably complex
filling of $S^{2q{+}1}$ which is built of two $(q{+}1)$-handles and a zero-handle. 
Thus the stable almost complex structure determines a unique complex structure
$J^\bullet$ on $W^\bullet = W \setminus \text{int} (D^{2q{+}2})$.  Using
the notation of \cite{Ding&Geiges} we have $\mathfrak{o}(W^\bullet ,J^\bullet) = 2$ and the formulae of
\cite[p.\ 3831]{Ding&Geiges} allow one to realise all possible
\emph{unstable} almost contact structures in any equivalence class of
\emph{stable} almost contact structures via connect sums with
$\partial W^\bullet$. 
\end{proof}


\subsection{The surgery lemmas} 
\label{subsec:key_surgery_lemmas}
In this subsection $\eta \colon B \to BU$ is again a stable complex
vector bundle.  The two lemmas below are consequences of the following classical theorem
of Wall:
\begin{Theorem}[\cite{Wall71}, Theorem~3]\label{thm:ofwall}
Suppose that $W$ is a (connected) $n$-dimensional 
cobordism from $X_-$ to $X_+$ and assume that $(W,X_-)$ is $r$-connected, 
that is, the relative homotopy groups $\pi _i (W, X_-)$ vanish for $i\leq r$. 
Suppose furthermore that $r\leq n-4$. Then,   
the cobordism is geometrically $r$-connected, that is, $W$ admits a 
(relative) handle decomposition built on $X_-$ which involves no handle with
index $\leq r$. \qed
\end{Theorem}
\noindent In particular, Wall's result converts a potentially subtle problem about Morse
 functions with certain properties to a question that is purely homotopy theoretic in nature.

\begin{Lemma}[Filling lemma] \label{lem:topological_filling}
	Let $(M, \acs)$ be an almost contact $(2q{+}1)$-manifold with
	induced stable complex structure $\scxs$ and 
        complex normal $(q-1)$-type $(B^{q{-}1}_{\scxs},
        \eta^{q{-}1}_{\scxs})$.
		If $q\geq2$, then
      	the following are equivalent:
\begin{enumerate}
\item
$(M, \acs)$ is the boundary of a compact almost complex
  $(2q{+}2)$-manifold $(W, \acxs)$ with handles only of index $q{+}1$
  and smaller.
\item
For any $\scxs$-compatible normal $(q{-}1)$-smoothing $\bscxs \colon M \to
B^{q{-}1}_\scxs$,
we have 
\[ [M, \bscxs] = 0 \in \Omega_{2q{+}1}(B^{q{-}1}_\scxs; \eta^{q{-}1}_\scxs). \]
\item
For some stable complex bundle $(B, \eta)$ and some $\scxs$-compatible normal
$(q{-}1)$-smoothing $\bscxs \colon M \to B$, we  have 
\[ [M, \bscxs] = 0 \in \Omega_{2q{+}1}(B; \eta). \]
\end{enumerate}
\end{Lemma}

\begin{proof}
	$(1) \Rightarrow (2)$: Suppose that $(W, J)$ is as in the
  statement of the lemma.  The almost complex structure $J$ defines a
  stable complex structure $\scxs_W \colon W \to BU$.  Let
  $(B^{q-1}_{\scxs_W}, \eta^{q-1}_{\scxs_W})$ be the complex normal $(q-1)$-type of 
  $(W, \scxs_W)$ and
  let $\bscxs_W \colon W \to B^{q-1}_{\scxs_W}$ be a $(q-1)$-smoothing of
  $B^{q-1}_{\scxs_W}$.  Let $i \colon M \to W$ be the inclusion.  Then the map
	\[ \psi: = \bscxs_W \circ i \colon M \to B^{q-1}_{\scxs_W} \]
defines a $(B^{q-1}_{\scxs_W}, \eta^{q-1}_{\scxs_W})$-structure on $M$
which is compatible with $\scxs = \nacs$ since
$J|_{\partial W} = \acs$.  Since the smooth manifold $W$ admits a handle
decomposition with handles only of index $(q{+}1)$ or less, by turning
such a decomposition upside down we conclude that $W$ has a handle
decomposition starting from $M$ and adding handles of dimension
$(q{+}1)$ and higher.  It follows that $i \colon M \to W$ is a
$q$-equivalence and hence the map $\xi \colon M \to B^{q-1}_{\scxs_W}$
is a $q$-equivalence.  Since $(B^{q-1}_{\scxs_W},
\eta^{q-1}_{\scxs_W})$ is the complex normal $(q-1)$-type of $W$, the
map $\eta^{q-1}_{\scxs_W} \colon B^{q-1}_{\scxs_W} \to BU$ is a
$q$-coequivalence.  It follows that $(B^{q-1}_{\scxs_W},
\eta^{q-1}_{\scxs_W})$ is a model for the complex normal $(q-1)$-type
of $(M, \scxs)$ and so we identify $(B^{q-1}_{\scxs},
\eta^{q-1}_{\scxs}) = (B^{q-1}_{\scxs_W}, \eta^{q-1}_{\scxs_W})$.  By
construction the map $\psi \colon M \to B^{q-1}_\scxs$ is a complex normal
$(q-1)$-smoothing and $(W, \bscxs_W)$ is a $(B^{q-1}_{\scxs},
\eta^{q-1}_{\scxs})$-null bordism of $(M, \psi)$.  It follows that
$[M, \psi] = 0 \in \Omega_{2q{+}1}(B^{q-1}_{\scxs};
\eta^{q-1}_{\scxs})$.  Now by Lemma \ref{lem:Aut(B^k)}, $[M, \bscxs] =
0 \in \Omega_{2q{+}1}(B^{q-1}_\scxs: \eta^{q-1}_\scxs)$ for any complex
normal $(q-1)$-smoothing $\bscxs \colon M \to B^{q-1}_\scxs$.
	
	$(2) \Rightarrow (3)$: Take $(B, \eta) = (B^{q-1}_{\scxs}, \eta^{q-1}_{\scxs})$.

	$(3) \Rightarrow (1)$: Let $(W, \bscxs_W)$ be a
        $B$-nullbordism of $(M, \bscxs)$. Using surgery below the
        middle dimension as in Proposition \ref{prop:Kreck_surgery},
       we may assume that $\bscxs_W \colon W
        \to B$ is a $(q{+}1)$-equivalence and in particular there are
        isomorphisms of fundmental groups $\pi = \pi_1(M) \cong
        \pi_1(B) \cong \pi_1(W)$.  If $i \colon M \to W$ denotes, the
        inclusion, the commutative diagram
	  \[ 
\xymatrix{ \pi_i(M) \ar[dr]_{\bscxs_*} \ar[rr]^{i_*} & & \pi_i(W) \ar[dl]^{(\bscxs_W)_*} \\ &
  \pi_i(B), }
\] 
and the facts that $\bscxs \colon M \to B$ is a $q$-equivalence and
$\bscxs_W \colon W \to B$ is a $(q{+}1)$-equivalence show that the
inclusion $i \colon M \to W$ is a $q$-equivalence.  By Theorem \ref{thm:ofwall}, it follows
that $W$ is diffeomorphic to a manifold obtained from $M$ by attaching
handles in dimension $(q{+}1)$ and higher.
		
		Turning the above handle decomposition upside down,
               we see that $W$ has a handle
                decomposition consisting of $k$-handles with $k \leq q{+}1$.
                Moreover, the $(B, \eta)$-strucure on $W$
                defines a stable complex structure $\scxs_W$ on $W$.
                Applying Lemma \ref{lem:destabilising} to the
                handlebody decomposition of $W$ we deduce that $W$
                admits an almost complex structure $\acxs$ such that
                $S \acxs = \scxs_W$.  The almost complex structure
                $\acxs$ induces some almost contact structure $\acxs|_{\partial W}$ on $M$ such that $S \acxs|_{\partial W} = S \acs$.  It
                follows that $\acs = \acxs|_{\partial W}+ \acs_0$ where
                $\acs_0 \in \pi_{2q{+}1}(SO(2q{+}1)/U(q))$ is a stably
                trivial almost contact structure on $S^{2q{+}1}$.  By
                Lemma \ref{lem:stable_to_unstable} the almost contact
                manifold $(S^{2q{+}1}, \acs_0)$ admits a Stein filling
                $(W_0, \stein_0)$ and in particular an almost complex
                filling $(W_0, \acxs_0)$.  It follows that the
                boundary connected sum $(W \natural W_0; \acxs
                \natural \acxs_0)$ is an almost complex filling of
                $\acxs|_{\partial W} + \acs_0 = \acs$.
\end{proof}

The above result admits a `relative' version, where consider bordisms
between two smoothings:

\begin{Lemma}[Stable surgery Lemma] \label{lem:stable_bordism}
Let $(W, \bar \scxs_W; M_0, M_1)$ be a $(B, \eta)$-bordism between
normal $(q-1)$-smoothings $(M_0, \bscxs_0)$ and $(M_1, \bscxs_1)$ of
dimension $2q{+}1 \geq 5$.  Then for $j=0, 1$ the bordism $W$ admits a
handlebody decomposition relative to $M_j$ consisting of handles of
index $k \leq q{+}1$.
\end{Lemma}

\begin{proof}
Let $i_j \colon M_j \to W$, $j = 0, 1$ denote the inclusion maps.
Using surgery below the middle dimension as in Proposition \ref{prop:Kreck_surgery}, we may assume that $\bscxs_W \colon W \to B$ is a
$(q{+}1)$-equivalence.  Now consider the following commutative diagram
\[ \xymatrix{ \pi_i(M_0) \ar[drr]_{(\bscxs_0)_*} \ar[rr]^{(i_{0})_*} & & \pi_i(W) 
\ar[d]^(0.4){(\bscxs_W)_*} & & \pi_i(M_1) \ar[ll]_{(i_1)_*} \ar[dll]^{(\bscxs_1)_*} \\ & & \pi_i(B). } \] 
Since the maps $\bscxs_i \colon M_i \to B$ are $q$-equivalences and 
$\bscxs_W \colon W \to B$ is a
$(q{+}1)$-equivalence, it follows that each inclusion $i_j \colon M_j \to
W$ is a $q$-equivalence.  By Theorem \ref{thm:ofwall}, $W$ admits a
handlebody decomposition relative to $M_{j+1}$ consisting of handles
of index $k' \geq q{+}1$.  If we turn this handbody
decomposition upside down we obtain a handlebody decomposition of $W$
relative to $M_j$ consisting of handles of index $k \leq q{+}1$.
\end{proof}

We next give the unstable version of the previous lemma:

\begin{Lemma}[Unstable surgery Lemma] \label{lem:unstable_bordism}
Let $(M_0, \acs_0)$ and $(M_1, \acs_1)$ be almost contact manifolds of
dimension $2q{+}1\geq 5$ with associated stable complex structures $\scxs_0$
and $\scxs_1$.  Suppose for $i = 0, 1$, that $\bscxs_i \colon M_i \to B$, are
$\scxs_i$-compatible normal $(q-1)$-smoothings in a stable complex
bundle $(B, \eta)$ which are $(B, \eta)$-bordant.
Then there
is an almost complex bordism $(W, \acxs; (M_0, \acs_0), (M_1,
\acs_1))$ between $(M_0, \acs_0)$ and $(M_1, \acs_1)$ such that for
$j=0, 1$ the manifold $W$ admits a handlebody decomposition relative
to $M_j$ consisting of handles of index $k \leq q{+}1$.
\end{Lemma}

\begin{proof}
	Let us give the proof for $j=0$, the proof for $j=1$ is
        similar.  By Lemmas
        \ref{lem:destabilising} and \ref{lem:stable_bordism} there is an almost complex bordism
        $(W, \acxs; (M_0, \acs_0), (M_1, \acs_1))$ where $W$ is
        obtained from $M_0$ by attaching handles of index $(q{+}1)$ or
        less and where the almost contact structure $\acs_1'$
        satisfies $S\acs_1' = S\acs_1$.  It follows that $\acs_1 =
        \acs_1' + \acs_0$ where $\acs_0 \in \pi_{2q{+}1}(SO(2q{+}1)/U(q))$
        is a stably trivial almost contact structure on $S^{2q{+}1}$.
        By Lemma \ref{lem:stable_to_unstable} the almost contact
        manifold $(S^{2q{+}1}, \acs_0)$ admits an almost complex filling
        $(W_0, \acxs_0)$ with handles of index $\leq q{+}1$.  Taking
        the boundary connected sum of $W$ and $W_0$ at the $M_1$
        boundary component of $W$ we obtain an almost complex bordism
        $(W \natural W_0, \acxs \natural \acxs_0; (M_0, \acs_0), (M_1,
        \acs_1))$ where $W \natural W_1$ has a handlebody
        decomposition relative to $M_0$ consisting of handles of index
        $(q{+}1)$ or less.
\end{proof}


\subsection{Complex normal $k$-types} 
\label{subsec:normal_k_types}
In this subsection, we identify the complex normal $k$-type
$(B^k_\scxs, \eta^k_\scxs)$, of a general stably complex manifold $(X,
\scxs)$ under certain assumptions for $k=1$ and $k=2$.  These
computations will play crucial roles in our applications
(cf. Section~\ref{sec:7_manifolds}).  We shall use the following
notation.  Since we do not distinguish between stable complex bundles
and their classifying maps, we shall write $f \oplus g \colon X \times
Y \to BU$ for the exterior Whitney sum of stable complex bundles
classified by maps $f \colon X \to BU$ and $g \colon Y \to BU$.  Also,
we let $\pi_{SU}
\colon BSU \to BU$ be the map of classifying spaces induced by the 
inclusion $SU \subset U$.

\begin{Lemma}  \label{lem:1type}
Let $(X, \scxs)$ be a stably complex manifold with $\pi = \pi_1(X)$.
\begin{enumerate}
\item \label{it:c1_onto}
If $\scxs_* \colon \pi_2(X) \to \pi_2(BU)$ is onto then 
\[ (B^1_\scxs, \eta^1_\scxs) = \bigl( K(\pi, 1) \times BU, {\rm pr}_{BU} \bigr).  \]

\item \label{it:c1_zero}
If $c_1(\scxs) = 0 \in H^2(X)$ then
\[ (B^1_\scxs, \eta^1_\scxs) = 
\bigl( K(\pi, 1) \times BSU, \pi_{SU} \circ {\rm pr}_{BU} \bigr).     \]
\end{enumerate}
\end{Lemma}

\begin{proof}
Both ${\rm pr}_{BU}$ and $\pi_{SU} \circ {\rm pr}_{BU}$ are
$2$-coequivalences.  Thus, from the defining properties of the second
Postnikov approximation of $\scxs \colon X \to BU$, it suffices to
find maps $\bscxs \colon X \to B^1_\scxs$ which are $2$-equivalences
and which factor $\scxs$ over $\eta^1_\scxs$.

\eqref{it:c1_onto}
Let $u \colon X \to K(\pi, 1)$ classify the universal covering of $X$ and 
define $\bscxs$ by
\[  \bscxs : = (u \times \scxs) \colon X \to K(\pi, 1) \times BU. \]
The assumption that $\scxs_*$ in onto on $\pi_2$ ensures that $\scxs$
is a $2$-equivalence and clearly ${\rm pr}_{BU} \circ \bscxs = \scxs$.

\eqref{it:c1_zero} Since $c_1(\scxs) = 0$, there is a lift of $\scxs$
to $\scxs' \colon X \to BSU$.  Define $\bscxs$ by
\[  \bscxs : = (u \times \scxs') \colon X \to K(\pi, 1) \times BSU. \] 
Since $\pi_2(BSU) = 0$, $\scxs$ is a $2$-equivalence and clearly 
$\pi_{SU} \circ {\rm pr}_{BU} \circ \bscxs = \scxs$.
\end{proof}

Now we consider the complex normal $2$-type of $(X, \scxs)$.  Let
$p_2 \colon X \to P_2(X)$ denote a $3$-equivalence from $X$ to its
second Postnikov stage, $P_2(X)$.

\begin{Lemma} \label{lem:2type}
Let $(X, \scxs)$ be a stably complex manifold and let $\gamma_\scxs$
by the unique complex line bundle over $P_2(X)$ such that 
$c_1(p_2^*(\gamma_\scxs)) = -c_1(\scxs)$.  Then
\[ (B^2_\scxs, \eta^2_\scxs) = 
\bigl( P_2(X) \times BSU, \gamma_\scxs \oplus \pi_{SU} \bigr). \]
\end{Lemma}

\begin{proof}
By definition, the map on second cohomology induced by $p_2$ is an
isomorphism:
$ p_2^* \colon H^2(P_2(X)) \cong H^2(X).$
Hence there is a (unique isomorphism class of) line bundle
$\gamma_\scxs$ over $P_2(X)$ such that $p_2^*(\gamma_X) =
-c_1(\scxs)$.  The stable complex bundle $\xi : = \scxs \oplus
p_2^*(\gamma_\scxs)$ satisfies
\[ c_1(\xi) = c_1(\scxs) - c_1(\scxs) = 0 \in H^2(X), \]
and so $\xi$ admits an $SU$-structure classified by a map $\xi' \colon
X \to BSU$. We define $\bscxs$ by
\[ \bscxs : = (p_2 \times \xi') \colon X \to P_2(X) \times BSU. \]
Since $BSU$ is $3$-connected and $\pi_3(P_2(X))= 0$, $\bscxs$ is a
$3$-equivalence. By construction we have $(\gamma_\scxs \oplus \pi_{SU})
\circ \bscxs = \scxs$ and clearly $\gamma_\scxs \oplus \pi_{SU}$ is a
$3$-coequivalence.  It follows that $(B^2_\scxs, \eta^2_\scxs)$ is
the complex normal $2$-type of $(X, \scxs)$.
\end{proof}

\section{Contact structures and complex normal bordism} \label{sec:complex}
After recalling the necessary definitions and the statement of
Eliashberg's $h$-principle, we state our main surgery theorems,
Theorems~\ref{thm:Stein} and \ref{thm:bordism}.  The proofs of these
theorems rest on the discussion presented in Section~\ref{subsec:key_surgery_lemmas}.

\subsection{Symplectic fillability and contact surgery} 
\label{subsec:sympfill_and_ctct_surg} 
Recall that a symplectic manifold $(W, \omega)$ is a
$(2q{+}2)$-dimensional manifold $W$ with a closed $2$-form $\omega$ such
that $\omega^{q{+}1} \neq 0$ at every point in $W$. In particular, a symplectic
manifold carries a canonical orientation. Recall, furthermore, that a
cooriented, codimension-1 distribution $\xi$ on a $(2q{+}1)$-manifold
$M$ is a \emph{contact structure} if there is a $1$-form $\alpha$ such that
$\ker (\alpha) = \xi$ and
\[ \alpha \wedge (d \alpha)^{q} \neq 0. \]
Note that this then also determines an orientation of $M$. 
Two contact manifolds $(M_0,\xi_0)$
and $(M_1,\xi_1)$ are \emph{contactomorphic} if there is a diffeomorphism
$\phi: M_0 \to M_1$ such that
$$\phi_*(\xi_0) = \xi_1.$$
We now recall the various notions of fillability for contact structures.
\begin{Definition}[Strongly symplectically fillable and exactly fillable]
A contact manifold $(M,\xi)$ is called strongly symplectically
  fillable if it bounds a compact symplectic manifold $(W,\omega)$
and there is an outward pointing vector field $V$ near $\del X$ such
that the Lie derivative satisfies $L_V\omega = \omega$, and $\lambda
=\iota_V\omega$ is a defining $1$-form for $\xi$.  If the symplectic
form $\omega$ is exact then we say that $(M,\xi)$ is exactly
  fillable.
\end{Definition}
\noindent A further specialisation of the notion of fillability is that
of Stein fillability. Recall that a Stein domain is a compact,
complex manifold $(W,J)$ with boundary that admits a function $\phi: W
\to [0,1]$ so that $\omega = -dd^{\mathbb {C}} \phi$ is a symplectic
form and $\phi^{-1}(1) = \del W$ is a regular level.

\begin{Definition}[Stein fillable]
A contact manifold $(M,\xi)$ is called Stein fillable if it 
bounds a Stein domain $(W,J)$ such that $\xi = J(TM) \cap TM$.
\end{Definition}
\noindent These notions of fillability fit into the following sequence
of inclusions of contactomorphism classes of contact
manifolds:
\begin{equation} \label{eq:contactflavours} 
\{\text{Stein fillable\} $\subseteq$ \{exactly fillable\} $\subseteq$
  \{strongly fillable\}. }
\end{equation}

\noindent 
\dcedit{A $k$-sphere $S^k \subset M$} in an contact manifold $(M^{2q{+}1}, \xi)$
is called {\em isotropic} if $TS^k \subset \xi$.
Surgery on an isotropic sphere $S^{k}$ 
can be performed in a way that is compatible with the
contact structure. If $k \leq q$ and $2q{+}1 \geq 5$ then any embedded sphere 
can be realised by an isotropic sphere and such surgeries can be realised by the
attachment of a symplectic or ``Weinstein'' $(k{+}1)$-handle $h_{k{+}1} :=
D^{k{+}1}\times D^{2q{+}1-k}$, provided that the associated almost complex
structure on the product manifold $M^{2q{+}1}\times [0,1]$ extends over
the trace $(M^{2q{+}1}\times [0,1]) \cup h_{k{+}1}$ of the
surgery (cf.\,\cite{Eliashberg??, Cieliebak&Eliashberg12}) . Furthermore, 
the symplectic nature of the handle attachment
shows that the symplectic fillability of a contact structure is
preserved under such contact surgeries.  In addition, Eliashberg
showed that when attaching a Weinstein handle to a Stein manifold, the
Stein structure also extends. (For more details concerning contact
surgery and Weinstein handles we refer the reader to
\cite{Cieliebak&Eliashberg12} or \cite{Geiges08} or \cite{Weinstein}.)

\begin{Theorem}\label{thm:h-principle}
Let $(M^{2q{+}1},\xi)$ be a contact manifold of dimension $2q{+}1 \geq 5$
with associated almost contact structure $\acs$.
Suppose that $k \leq q$ and that $(M', \acs')$ is obtatined from $(M, \acs)$ 
via a $k$-dimensional
almost complex surgery with trace $(M \times I) \cup h_{k{+}1}$ 
as in Defintion \ref{def:almost_complex_surgery}.
Then $M'$ admits a contact structure $\xi'$.  If is $(M^{2q{+}1},\xi)$ symplectically 
or exactly fillable, then so is $(M',\xi')$.  Moreover, if $(W,J)$ is a Stein filling of
$(M^{2q{+}1},\xi)$ then there is $J'$ on $W \cup h_{k{+}1}$ such that it
is a Stein filling of $(M',\xi')$. \qed
\end{Theorem}

\begin{Remark}
Although \cite[Theorem 6.3.1]{Geiges08}  is not stated explicitly for exact 
fillability, the proof also holds in the case of exact fillability, since 
attaching Weinstein handles does not affect the exactness of the symplectic 
form on the filling. 
\end{Remark}

\noindent Applying Theorem \ref{thm:h-principle} inductively over a
handle decomposition, one obtains the following
(cf.\ \cite{Cieliebak&Eliashberg12}, Theorem 8.15):

\begin{Corollary}[Eliashberg's $h$-principle]\label{cor:h-principle}
	Let $(W,J)$ be a compact $(2q{+}2)$-dimensional almost complex
        manifold with handles only in dimensions $q{+}1$ or less.  Then
        $J$ is homotopic to an almost complex structure $\tilde{J}$ so
        that $(W,\tilde{J})$ is a Stein filling of $M = \partial W$
        and in particular, $M$ is Stein fillable. \qed
\end{Corollary}


\subsection{Surgery theorems} 
\label{subsec:surgery_theorems}
In this subsection we state our main theorems concerning Stein
fillings and contact surgery.  The results will be mainly translations
of the surgery theoretic results from
Section~\ref{subsec:key_surgery_lemmas}.  We begin with the result
corresponding to Lemma~\ref{lem:stable_to_unstable}.

\begin{Lemma} \label{lem:stable_to_unstable_Stein}
	Suppose that the almost contact manifold $(M, \acs)$ can be realised as a
        contact structure $\xi$.  Then every homotopy class of almost
        contact structure which is stably equivalent to $\acs$ admits
        a contact structure obtained from $\xi$ by connected sum with
        a Stein fillable contact structure on $S^{2q{+}1}$.
\end{Lemma}
\begin{proof}
The proof is a simple combination of the proof of 
Lemma~\ref{lem:stable_to_unstable} and Corollary~\ref{cor:h-principle}:
the almost contact structures found on $S^{2q{+}1}$ in the proof of
Lemma~\ref{lem:stable_to_unstable} are Stein fillable contact
structures, and the boundary connect sum of two Stein 
fillings is a Stein filling.  
\end{proof}

Using the notation and terminology of Section \ref{sec:appendix}, we
obtain the following bordism characterisation of Stein fillability,
proving (an expanded version of) Theorem~\ref{thm:main}:

\begin{Theorem} [Filling Theorem] \label{thm:Stein}
Let $(M, \varphi)$ be a closed almost contact $(2q{+}1)$-manifold with
induced stable complex structure $\scxs$ and complex normal $(q{-}1)$-type $(B^{q{-}1}_{\scxs},
\eta^{q-1}_{\scxs})$. If $q\geq2$, then the following are equivalent:
\begin{enumerate}
\item
$(M, \varphi)$ admits a Stein-fillable contact structure;
\item
for any $\scxs$-compatible normal $(q{-}1)$-smoothing ${\bar {\zeta}} \colon M \to
B^{q{-}1}_{\zeta}$, we have
\[ 
[M, {\bar {\zeta}}] = 0 \in \Omega_{2q{+}1}(B^{q{-}1}_{\zeta}; \eta^{q{-}1}_{\zeta});
\]
\item
for some stable complex bundle $(B, \eta)$ and some
$\scxs$-compatible normal $(q{-}1)$-smoothing ${\bar {\zeta}} \colon M \to B$, we have
\[ [M, {\bar {\zeta}}] = 0 \in \Omega_{2q{+}1}(B; \eta). \]
\end{enumerate}
\end{Theorem}

\begin{proof}
For stable almost contact structures this is just a combination of Lemma
\ref{lem:topological_filling} and Eliashberg's $h$-principle
(cf.\ Corollary \ref{cor:h-principle}). Lemma
\ref{lem:stable_to_unstable_Stein} then implies that any almost contact
structure in a given stable class can be realised as a Stein fillable
contact structure, as soon as one can.
\end{proof}
Similar arguments provide

\begin{Theorem} [Surgery Theorem]\label{thm:bordism}
Let $(M_0, \acs_0)$ and $(M_1, \acs_1)$ be almost contact manifolds of
dimension $2q{+}1\geq 5$ with associated stable complex structures $\scxs_0$
and $\scxs_1$.  Suppose for $i = 0, 1$, that $\bscxs_i \colon M_i \to B$ are
$\scxs_i$-compatible normal $(q-1)$-smoothings in a stable complex
bundle $(B, \eta)$ such that
\[ [M_0, \bscxs_0] = [M_1, \bscxs_1] \in \Omega_{2q{+}1}(B; \eta).\]
  Then $(M_0, \acs_0)$ admits a contact structure
  if and only if $(M_1, \acs_1)$ does. Moreover, $(M_0, \acs_0)$ admits a
  fillable contact structure in any sense (cf.\
  display~\eqref{eq:contactflavours} above) if and only if $(M_1, \acs_1)$
  does.
\end{Theorem}

\begin{proof}
	By Lemma \ref{lem:stable_bordism}, there is a $(B,
        \eta)$-bordism $(W, \bar \scxs_W)$ between $(M_0, \bscxs_0)$
        and $(M_1, \bscxs_1)$ such that $(W, \bar \scxs_W)$ is
        obtained from $(M_i, \bar \zeta _i) \times [0, 1]$, $i = 0, 1$,
        by attaching $k$-handles, $k \leq q{+}1$, over which the almost
        complex structure extends.  The result now follows from
        Theorem~\ref{thm:h-principle} above. This then gives contact
        structures in the desired stable class of almost contact
        structures. However, by Lemma \ref{lem:stable_to_unstable_Stein} one
        can then realise all almost contact structures via connected
        sum with certain contact structures on spheres. As all these
        contact structures are Stein fillable, this does not affect
        the fillability of the contact structures.
\end{proof}
\begin{Remark}
  The idea of constructing contact structures via surgery techniques is not
  new, and Geiges and Thomas, in particular, have employed such methods to
  prove the existence of contact structures under various topological
  assumptions. Indeed, using the explicit description of normal $1$-types
  given in Lemma \ref{lem:1type}, one can deduce the Bordism Theorem of
  \cite{Geiges01} as a special case of Theorem \ref{thm:bordism}. The main
  benefit of Theorem \ref{thm:bordism} is that it provides a unified approach
  to this point of view without making any assumptions on the almost contact
  structures involved.  
\end{Remark}

In the following sections we will use the Filling Theorem above to produce 
Stein fillable contact structures and obstructions to Stein fillability.
The Surgery Theorem, on the other hand, is useful
for finding contact structures on manifolds which 
cannot carry Stein fillable structures as we now explain.
Let $\beta$ denote a class of contact structures which is closed under
Weinstein handle attachment and which includes Stein fillable
contact structures; for example $\beta$ could be the class of 
symplectically  fillable contact structures.  We define
\[ \Omega^\beta_{2q{+}1}(B; \eta) \subset \Omega_{2q{+}1}(B; \eta) \]
to be the set of bordism classes with representatives $\bscxs \colon N \to B$
such that $\bscxs$ is $\scxs$-compatible and such that $(N, \scxs)$ admits
a contact structure $\xi$ in the class $\beta$.  
We emphasise that here we 
make no connectivity assumption on the map $\bscxs \colon N \to B$.

\begin{Corollary} \label{cor:alpha-bordism}
Let $(M, \acs)$ be an almost contact $(2q{+}1)$-manifold with associated
stable complex structure $\scxs$ and let $(B, \eta)$ be a stable
complex bundle.  If $q \geq 2$, the map $\bscxs \colon M \to B$ is a
$\scxs$-compatible normal $(q-1)$-smoothing and
\[  [M, \bscxs] \in \Omega_{2q{+}1}^\beta(B; \eta), \]
then $(M, \scxs)$ admits a contact structure in the class $\beta$.
\end{Corollary}

\begin{proof}
By assumption, there is a contact manifold $(N, \xi)$ with associated stable complex
structure $\scxs_N$ and with a $\scxs_N$-compatibe $(B, \eta)$-structure
$\bscxs_N \colon N \to B$ such that $[N, \bscxs_N] = [M, \bscxs] \in \Omega_{2q{+}1}(B; \eta)$.
By Proposition \ref{prop:Kreck_surgery}, we may perform $(B, \eta)$-surgeries of dimension 
$q$ or less on $\bscxs \colon N \to B$ to obtain a $(q-1)$-smoothing 
$\bscxs_{N'} \colon N' \to B$, with induced stable complex structure $\scxs'$ say.  
By Lemmas \ref{lem:destabilising} and \ref{lem:stable_to_unstable} and Theorem \ref{thm:h-principle}, 
$(N', \acs')$ admits a contact structure $\xi'$ in the class $\beta$ and with 
associated almost contact structure $\acs'$ which stabilises to $\scxs'$.
Applying Theorem \ref{thm:bordism} to $(N', \acs')$ and $(M, \acs)$,
we deduce that $(M, \acs)$ admits a contact structure in the class $\beta$.
\end{proof}

The line of reasoning from the proof of Corollary \ref{cor:alpha-bordism}
was pursued in \cite{BCSS2} for $\beta$ the class of all contact structures.
There contact structures on manifolds of the form $M\times S^2$ 
with $M$ contact were shown to exist by 
finding Stein cobordisms from $M\times T^2$ and applying a result of 
Bourgeois \cite{bourgeois} which provides a contact structure for this latter manifold once 
$M$ is contact. 
More generally, Corollary \ref{cor:alpha-bordism}
gives a framework for approaching the 
symplectic version of the Stein Realisation Problem~\ref{prob:SteinReal}:

\begin{Problem}[Symplectic Relaization Problem]
Determine which almost contact structures on a given manifold can be
realised by strongly/exactly fillable contact structures. Does the
answer depend on whether one considers strong or exact fillings?
\end{Problem}


\section{Simply connected $7$-manifolds} 
\label{sec:7_manifolds}
As an application of the methods developed in
Section~\ref{sec:complex}, we now give a proof of
Theorem~\ref{thm:seven_mfold_contact}.  The proof will show that the
Stein fillability obstruction of Theorem~\ref{thm:Stein} vanishes by
showing that the relevant bordism group is itself
trivial.

Before turning to the computation of the bordism group, however, we show that every
7-manifold considered in Theorem~\ref{thm:seven_mfold_contact} admits an
almost contact structure.

To start the argument, recall that a manifold $M$ admits a spin$^c$
structure, that is, a lift of the structure group of $TM$ from $SO(n)$
to the group $Spin ^\C(n)$, if and only if the second Stiefel-Whitney
class $w_2(M)\in H^2(M; \Z_2 )$ admits an integral lift. (The Lie
group $Spin ^\C(n)$ can be defined as the extension of $SO(n)$ by
$S^1$ with the property that $Spin ^\C(n)\to SO(n)$ is the unique
nontrivial principal $S^1$-bundle over $SO(n)$.)  Since each manifold
$M$ in Theorem \ref{thm:seven_mfold_contact} is simply connected and
has torsion free $\pi _2(M)\cong H_2(M)$, the mod 2 reduction map
$H^2(M)\to H^2 (M; \Z_2)$ is onto, and so $M$ admits a spin$^c$
structure.

Notice that $U(n)\subset SO(2n)$, and since any $S^1$-bundle over
$U(n)$ is trivial (by the fact that $H^2(U(n))=0$), we have that the
restriction of the bundle $Spin ^\C (2n) \to SO(2n)$ over $U(n)$ is
trivial.  Consequently $Spin ^\C (2n)$ contains $U(n)\times S^1$, so
in particular $U(n)$ embeds into $Spin ^\C (2n)$. This embedding
provides a homomorphism of topological groups $U\to Spin ^\C$. 
Similarly, $SU(n)$ embeds into
$SO(2n)$, and since $SU(n)$ is simply connected, this embedding lifts
to an embedding $SU(n)\to Spin (2n)$ (recall that $Spin (2n)\to
SO(2n)$ is the nontrivial double cover of $SO(2n)$). This construction
then provides a homomorphism of topological groups $SU\to Spin$.

We first show that every spin$^c$ structure on a $7$-manifold is
induced by an almost contact structure.

\begin{Lemma} \label{lem:7d_acs_existence}
A compact oriented $7$-manifold $X$ admits an almost contact structure
if and only if it admits a spin$^c$ structure.  Moreover, any spin$^c$
structure on $X$ is induced from some almost contact structure on $X$.
\end{Lemma}

\begin{proof}
\dcedit{
By Lemma~\ref{lem:stable_to_unstable}~\eqref{lem:stable_to_unstable:onto},
any stable complex structure $\scxs$ on
$X$ can be destabilised to an almost contact structure $\acs$.
Hence it is enough to show that $X$ admits a stable complex structure
if and only if $X$ admits a stable spin$^c$ structure, that is, a map
into $BSpin ^\C$ covering the map $X\to BSO$ given by the stable
tangent bundle.}
Since a stable complex structure on $X$ induces a
spin$^c$ structure on $X$, we only need to show that any stable spin$^c$ structure on
$X$ can be lifted to a stable complex structure.

The homomorphism of topological groups $U \to Spin^\C$ induces a map
of classifying spaces which gives a fibre bundle
\begin{equation} \label{eq:BU_to_BSpinC}
        Spin^\C/U \xra{i} BU \xra{} BSpin^\C.
\end{equation}
By Bott periodicity the quotient $Spin^\C/U$ is $5$-connected and
$\pi_6(Spin^\C/U) \cong \Z.$ Suppose that $\theta \colon X \to BSpin^\C$ is a
spin$^c$ structure on $X$.  We must show that the following lifting
problem has a solution:
\[  \xymatrix{ & BU \ar[d] \\
X \ar@{.>}[ur]\ar[r]^(0.35)\theta & BSpin^\C.} \]
Since $Spin^\C/U$ is $5$-connected, the primary obstruction to lifting
$\theta$ is a cohomology class $\theta^*(\alpha) \in H^7(X)$, where
we have identified $\pi_6\ (\Spin^\C/U)$ with $\Z$ and the universal
obstruction class $\alpha \in H^7(BSpin^\C)$ is defined below.
We shall show that $2 \alpha = 0$.  Since $H^7(X)$ is torsion
free, it follows that $\theta^*(\alpha) = 0$ and hence $\theta$ lifts
to a stable complex structure on $X$.

It remains to define $\alpha$ and to prove that $2 \alpha = 0$.  Let
$x \in \pi_6(Spin^\C/U) \cong \Z$ be a generator.  Since $Spin^\C/U$
is $5$-connected, there is a generator $\wh x \in H^6(Spin^\C/U)$
such that $\an{\wh x, \rho(x)} = 1$  where $\rho \colon
\pi_6(Spin^\C/U) \to H_6(Spin^\C/U)$ is the Hurewicz homomorphism.
The class $\wh x$ is transgressive in the Leray-Serre cohomology
spectral sequence of the fibration \eqref{eq:BU_to_BSpinC}, and we
define
\[ \alpha : = \tau(\wh x) \in H^7(BSpin^\C), \]
where $\tau \colon H^6(Spin^\C/U) \to H^7(BSpin^\C)$ is the
transgression homomorphism.  Since the kernel of $\tau$ is the image
of the homomoprhism $i^* \colon H^6(BU) \to H^6(Spin^\C/U)
\cong \Z$, it suffices to show that the image of $i^*$ is the subgroup
of index two.  Now $H^*(BU) = \Z[c_1, c_2, c_3, \dots]$ is the
polynomial algebra on the Chern classes and the composition $S^6
\xra{x} Spin^\C/U \xra{i} BU$ determines the stable complex vector
bundle $x^*i^*(EU)$ over $S^6$ where $EU \to BU$ is the universal
bundle.  By \cite[Chapter7, Corollary 9.9]{Husemoller94}, every complex bundle
$E$ over $S^6$ is such that $c_3(E) \in 2 \cdot H^6(S^6)$ and moreover there is a
complex bundle $E_0$ over $S^6$ where $c_3(E_0)$ is twice a generator
of $H^6(S^6)$.  It follows that $i^*(H^6(BU)) = 2 \cdot
H^6(Spin^\C/U)$ and the lemma follows.
\end{proof}

We now reduce the proof of Theorem~\ref{thm:seven_mfold_contact} to
the calculation of certain bordism groups.
Let $\acs$ be an almost contact structure on $M$, with associated
stable complex structure $\scxs$, let $H = H_2(M)$ and
let $\gamma$ be the complex line bundle over the Eilenberg-MacLane
space $K(H, 2)$ with $c_1(\gamma) = -c_1(\scxs) \in H^2(K(H, 2); \Z)
\cong H^2(M; \Z)$.  By Lemma \ref{lem:2type}, the complex normal 2-type
of the stably complex manifold $(M, \scxs)$ is
\[ (B^2_{\scxs}; \eta^2_{\scxs}) = (K(H, 2) \times BSU; \gamma \oplus \pi_{SU}) , \]
where $\pi_{SU} \colon BSU \to BU$ is the map induced by the inclusion $SU \to U$ and $\oplus$
denotes the exterior Whitney sum of complex bundles: for further details,
see Section~\ref{subsec:normal_k_types}.
It follows that there is an isomorphism of bordism groups
  \begin{equation}\label{eq:isom}
\Omega_7(B^2_{\scxs}; \eta^2_{\scxs}) \cong \Omega_7^{SU}(K(H, 2); \gamma),
\end{equation}
where the latter group is a certain $\gamma$-twisted $SU$-bordism
group of $K(H, 2)$.  This is the bordism group of triples $(N, f,
\alpha)$ where $N$ is a closed smooth manifold, $f \colon N \to K(H,
2)$ is a map and $\alpha$ is an $SU$ structure on the Whitney sum of
$f^*(\gamma)$ and the stable normal bundle of $N$.

The remainder of this subsection gives the proof of the following proposition.

\begin{Proposition} \label{prop:7bordism}
For any finitely generated free abelian group $H$ and for any complex
line bundle $\gamma$ over $K(H, 2)$, we have $\Omega_7^{SU}(K(H, 2);
\gamma) = 0$.
\end{Proposition}

We shall need the following result on $SU$-bordism groups.

\begin{Lemma} \label{lem:Stong}
If $k$ is not divisible by 4 then $\Omega_{2k{+}1}^{SU} = 0$.
\end{Lemma}
\begin{proof}
By \cite[p.\,117]{Stong}, we know that $\Omega_*^U = 0$ for any odd
dimension $* = 2k{+}1$.  Also by \cite[p.\,238]{Stong}, the kernel of the
forgetful homomorphism $\Omega_*^{SU} \to \Omega_*^U$ is the torsion
subgroup of $\Omega_{*}^{SU}$.  But by \cite[p.\,248]{Stong}, the
torsion subgroup of $\Omega_{2k{+}1}^{SU}$ vanishes if $k$ is not divisible by 4, concluding
the proof.
\end{proof}

For the case $H=0$ in Proposition \ref{prop:7bordism}, by
Lemma \ref{lem:Stong} implies that $\Omega_7^{SU} = 0$ as required.  Hence we
assume that $H$ is not the zero group.  We wish to compute the
$\gamma$-twisted $SU$-bordism of $K(H, 2)$.  A very similar situation
is discussed in \cite[Section 6]{Kreck-Stolz91} where Kreck and Stolz
compute certain twisted spin bordism groups of $K(\Z, 2)$.  
\dcedit{
Since the
Thom space of the exterior Whitney sum of bundles is homotopy
equivalent to the smash product of the individual Thom spaces, 
the Pontrjagin-Thom construction gives an isomorphism
\begin{equation} \label{eq:Thom_spectra}
\Omega_*^{SU}(K(H, 2); \gamma) \cong \pi_*(T(\gamma \oplus \pi_{SU})) \cong \pi_*(T(\gamma) \wedge MSU) \cong
\wt\Omega_*^{SU}(T(\gamma)) .
\end{equation}
}
\noindent Here $MSU$ is the Thom spectrum defined by special unitary boridsm,
$T(\gamma)$ is the Thom space of the bundle $\gamma$ over $K(H,
2)$, $\wedge$ denotes the smash product of spectra and $\wt
\Omega_*^{SU}$ denotes reduced special unitary bordism.  As a
consequence of \eqref{eq:Thom_spectra}, there is an Atiyah-Hirzebruch
spectral sequence (AHSS),
\[ E_{p, q}^2 = H_{p+2}(T(\gamma); \Omega_q^{SU}) \Longrightarrow
\Omega_{p+q}^{SU}(K(H, 2); \gamma), \]
which converges to  the associated graded object of a filtration on
$\Omega^{SU}_{p+q}(K(H, 2); \gamma)$.  By the Thom isomorphism, $\wt
H^*(T(\gamma))$ is a free module over $H^*(K(H, 2))$ with generator
the Thom class $U \in H^2(T(\gamma))$ of $\gamma$.  As a consequence,
$H_*(T(\gamma))$ vanishes in odd degrees.  Now, by Lemma
\ref{lem:Stong}, $\Omega_{2k{+}1}^{SU} = 0$ for $k =1,2,3$ and
$\Omega_1^{SU} \cong \Z_2$ by \cite[p.\,248]{Stong}. It follows
that the $7$-line of the $E^2$-page of the AHSS above has only one
non-vanishing term and that is
\[ E^2_{6, 1} = H_8(T(\gamma); \Omega_1^{SU})  \cong H_6(K(H,2);\Omega^{SU}_1). \]
We claim that $E^3_{6, 1} = 0$, which proves Proposition
\ref{prop:7bordism}.  To see that $E^3_{6,1} = 0$, we need to
understand the following differentials in the AHSS, where we use that fact that
$\Omega _2 ^{SU}\cong \Z /2$ by \cite[p.\,248]{Stong}:
\begin{equation} \label{eq:d2}
d^2_{8, 0} \colon H_{10}(T(\gamma)) \to H_8(T(\gamma); \Z_2) ~~
\text{and} ~~ d^2_{6, 1} \colon H_8(T(\gamma); \Z_2) \to H_6(T(\gamma);
\Z_2).
\end{equation}
Since the map $SU \to \Spin$ is a $6$-equivalence, these differentials
for $SU$-bordism will coincide with the corresponding differentials
for spin bordism. The differentials in the spin case have been
computed by Teichner \cite[Lemma 2.3.2]{Teichner92}.  Hence we have
the following lemma.

\begin{Lemma}[{\cite[Lemma 2.3.2]{Teichner92}}] \label{lem:d2}
	Let $\rho_2 \colon H_{*}(T(\gamma)) \to H_{*}(T(\gamma); \Z_2)$ be the
	homomorphism induced by reduction mod~$2$ and let
	$(Sq^2)^* \colon H_{*+2}(T(\gamma); \Z_2) \to H_*(T(\gamma); \Z_2)$ be the
	dual of the Steenrod squaring operation 
	$Sq^2 \colon H^*(T(\gamma); \Z_2) \to H^{*+2}(T(\gamma); \Z_2)$.  Then 
	the differentials in \eqref{eq:d2} above are given by
	\[d^2_{8, 0} = 
(Sq^2)^* \circ \rho_2 \quad \text{and} \quad d^2_{6, 1} = (Sq^2)^*. \]
        \qed
\end{Lemma}
	
The following lemma is equivalent to the claim that $E^3_{6, 1} = 0$ in the AHSS
and hence completes the proof of Proposition \ref{prop:7bordism}.

\begin{Lemma} \label{lem:d_2-is-exact}
	For all finitely generated free abelian groups $H$ and for all
	complex line bundles $\gamma$ over $K(H, 2)$ we have
	\[ \Ker(d^2_{6, 1}) = \im(d^2_{8, 0}) .\]
\end{Lemma}

\begin{proof}
	The lemma is trivial if $H = 0$, so we assume that $H$ in
        non-zero.  We give the proof by viewing the situation from the
        point of view of homological algebra over the field
        $\Z_2$. Recall that $T(\gamma)$ denotes the Thom space of
        $\gamma$.  We define a chain complex $(C_*(H, \gamma), d)$ by
        setting
	\[ C_i(H, \gamma) := H_{2i+2}(T(\gamma); \Z_2),  \textrm{ for } i \geq 0\]
	and defining the differential $d$ by 
	\[ d_{i+1} := (Sq^2)^* \colon H_{2i+4}(T(\gamma); \Z_2)) \to H_{2i+2}(T(\gamma); \Z_2).\]
To see that the differential satisfies $d^2 = 0$, we first recall the
Adem relation $Sq^2Sq^2 = Sq^3Sq^1$, which entails that $Sq^2Sq^2 = 0$
on $H^*(T(\gamma); \Z_2)$ since the non-zero mod~$2$ cohomology
groups of $T(\gamma)$ are concentrated in even degrees.  It follows
that $(Sq^2)^*(Sq^2)^*$, which is the dual of $Sq^2Sq^2$, vanishes.
	
	Since the homomorphism $\rho_2 \colon H_{10}(T(\gamma)) \to
        H_{10}(T(\gamma); \Z_2)$ is onto, to prove the lemma it
        suffices to show that the third homology group of the chain
        complex $(C_*(H, \gamma), d)$ vanishes:
	\[ H_3(C_*(H, \gamma), d) = 0 .\]
        In the case where $H \cong \Z$, it is a simple exercise using
        the Thom isomorphism to check that the homology of $(C_*(\Z, \gamma), d)$ is
        trivial if $w_2(\gamma) \neq 0$, and if $w_2(\gamma)
        = 0$ then
	\[ H_*(C_*(\Z, \gamma), d) \cong \left\{ \begin{array}{cl} \Z_2 & * = 0 \\
	0 & * > 0.	
	\end{array} \right.  \]
	
	We shall prove the general case by induction from these two
        cases. Let $\gamma$ be a complex line bundle over $K(H, 2)$. When $H$ 
		has rank greater than one, let $H =H_0\oplus
        \Z$ with the property that $c_1(\gamma )|_{H_0}=0$. (If
        $c_1(\gamma )=0$, then any decomposition of $H$ will do, if
        $c_1(\gamma )\neq 0$ then take $H_0:= \Ker(c_1(\gamma) \colon
        H \to \Z)$.)  Observe that there is a split short exact
        sequence $H_0 \to H \xra{\pi} \Z$, such that $\gamma \cong
        \pi^* \gamma'$ for the map $\pi \colon K(H, 2) \to K(\Z, 2)$
        and for some complex line bundle $\gamma '$ over $K(\Z, 2)$.
If $\underline{\C}_X$ denotes the trivial line bundle over a
        space $X$, then there is an isomorphism of complex
        vector bundles
	\[ \underline{\C}_{K(H, 2)} \oplus \gamma \cong \underline{\C}_{K(H_0, 2)} \oplus \gamma'. \]
        where the first $\oplus$ denotes the usual Whitney sum over
        $K(H, 2)$ and the second $\oplus$ the exterior Whitney sum
        over $K(H_0, 2) \times K(\Z, 2) = K(H, 2)$.  Since the Thom
        space of the exterior Whitney sum of bundles is homotopy
        equivalent to the smash product of the Thom spaces of each
        bundle,
	\[ T(\gamma) \wedge S^2 \simeq M(\underline{\C}_{K(H_0, 2)}) \wedge T(\gamma') .\]
	If $x \in H^*(M(\underline{\C}_{K(H_0, 2)}); \Z_2)$ and $y \in H^*(T(\gamma'); \Z_2)$ and 
	$x \wedge y$ denotes their exterior cup product in 
	$H^*(M(\underline{\C}_{K(H_0, 2)}) \wedge T(\gamma'); \Z_2)$,
	then the Cartan formula for $Sq^2$ gives
	\[ Sq^2(x \wedge y) = Sq^2 x \wedge y + Sq^1 x \wedge Sq^1 y + x \wedge Sq^2 y = 
	Sq^2 x \wedge y + x \wedge Sq^2 y\]
	since $Sq^1x$ and $Sq^1y$ have odd degree and are thus zero. 
		Thus there is an isomorphism of chain complexes
	\[ (C_*(H_0 \oplus \Z, \gamma), d) \cong (C_*(H_0, 0), d) 
	\otimes (C_*(\Z, \gamma'), d), \]
where $\otimes$ denotes the tensor product of chain
        complexes.  Applying the Kunneth theorem for the homology
        groups of a tensor product of chain complexes over $\Z_2$
        inductively gives us that $H_3(C_*(H, \gamma), d) = 0$ for all
        groups $H$ and all complex line bundles $\gamma$.  This
        completes the proof of the lemma.
\end{proof}

\begin{proof}[Proof of Theorem~\ref{thm:seven_mfold_contact}]
Since a closed, oriented, simply connected manifold $M$ with torsion
free $\pi _2 (M)$ admits a spin$^c$ structure,
Lemma~\ref{lem:7d_acs_existence} implies the existence of an almost
contact structure $\acs$ on $M$.  By Equation~\eqref{eq:isom},
Proposition~\ref{prop:7bordism} implies $\Omega_7(B^2_{\scxs};
\eta^2_{\scxs}) = 0$ where $\scxs = \nacs$.  It follows that for 
any $\scxs$-compatible normal $2$-smoothing $\bscxs \colon M
\to B^2_{\scxs}$, we have $[M, \bscxs] = 0 \in\Omega_7(B^2_{\scxs};
\eta^2_{\scxs})$.  By Theorem \ref{thm:Stein} the almost contact
manifold $(M, \acs)$ is then Stein fillable.
\end{proof}

%

\section{Stein fillings of homotopy spheres} 
\label{sec:homotopy_spheres}
Recall that an $n$-dimensional homotopy sphere is a closed, smooth, oriented
manifold $\Sigma$ which is homotopy equivalent to $S^n$.
The set of oriented diffeomorphism classes of homotopy $n$-spheres
forms an abelian group $\Theta_n$ under the operation of connected sum:
\[ \Theta_n : = \{ \, [\Sigma] \,| \, \Sigma \simeq S^n \}.\]
For $n \geq 5$, every homotopy $n$-sphere $\Sigma$ is homeomorphic to $S^n$
\cite{Smale61}, 
hence $\Theta_n$ may be regarded as the group of oriented
diffeomorphism classes of smooth structures on the $n$-sphere.
 
We now recall some fundamental facts about the group $\Theta_n$ proved
by Kervaire and Milnor.  For further information, we refer the reader
to \cite{Kervaire-Milnor63, Levine85} and \cite[6.6]{Lueck}.
Let $O$ denote the stable orthogonal group, $\pi_n^S$ the $n^{th}$ stable
homotopy group of spheres and recall the $J$-homomorphism
\[ J_n \colon \pi_n(O) \to \pi_n^S. \]
Since $\pi_n^S$ is a finite group, the cokernel of $J_n$, $\Coker(J_n)$, is also
finite.  
We state the following theorem of Kervaire and Milnor only for the case of 
interest to us where $n = 2q{+}1 \geq 5$.

\begin{Theorem}[{\cite[Section 4]{Kervaire-Milnor63}}, {\cite[Theorem 6.6]{Kervaire-Milnor63}}] \label{thm:KM}
For $2q{+}1 \geq 5$ the abelian group $\Theta_{2q{+}1}$ lies in a short exact
sequence 
\[ 0 \dlra{} bP_{2q{+}2} \dlra{} \Theta_{2q{+}1} \dlra{\eta} \Coker(J_{2q{+}1}) \dlra{} 0 \]
where $bP_{2q{+}2}$ denotes the finite cyclic group of homotopy $(2q{+}1)$-spheres which bound
parallelisable manifolds. \qed
\end{Theorem}

When we move to the stable complex setting, we have the following

\begin{Example} \label{ex:k-type_of_sphere}
Every homotopy $(2q{+}1)$-sphere $\Sigma$ is stably
parallelisable by \cite[Theorem 3.1]{Kervaire-Milnor63} and
hence admits an almost contact structure $\acs$ with
stabilisation $\scxs := \nacs$. 
The complex normal $(q-1)$-type of $(\Sigma, \scxs)$ is
independent of the choice of $\acs$ and is given by
\[ 
(B^{q-1}_{\scxs}, \eta^{q-1}_{\scxs}) = (BU\an{q{+}1}, \pi_{q{+}1}), 
\]
where $\pi_{q{+}1} \colon BU\an{q{+}1} \to BU$ is the $q^{th}$
connective cover of $BU$: see Example \ref{ex:Connective_complex_bordism}.
\end{Example}

For homotopy spheres bounding parallelisable manifolds we have the following
well-known proposition.

\begin{Proposition} \label{prop:bP=Stein} 
Every homotopy sphere $\Sigma \in bP_{2q{+}2}$ is Stein fillable.
\end{Proposition}

\begin{proof}
In order to exhibit an explicit Stein filling for $\Sigma$, we use the fact that every $\Sigma \in bP_{2q{+}2}$ 
is diffeomorphic to a `Brieskorn sphere', \cite[Korollar 2]{Brieskorn66}.
That is, $\Sigma \cong \Sigma(a_1,a_2,\ldots,a_{q{+}2})$ is realised as the intersection of the singular hypersurface 
\[ H_0 = \{(z_1, \dots, z_{2q{+}2}) \, | \, z^{a_1}_1 +  z^{a_2}_2 + \ldots + z^{a_{q{+}2}}_{q{+}2}= 0\} \subset \mathbb{C}^{q{+}2} \]
with the unit sphere $S^{2q + 3} \subset \mathbb{C}^{q{+}2}$ for suitable $a_i \in \N$. A Stein filling is then given 
by considering the part of a regular hypersurface, 
\[ H_{\epsilon} = \{(z_1, \dots, z_{2q{+}2}) \, | \, z^{a_1}_1 +  z^{a_2}_2 + \ldots + z^{a_{q{+}2}}_{q{+}2}= \epsilon\}, \]
that intersects the unit ball $B^{2q + 4} \subset \mathbb{C}^{q{+}2}$
for any small $\epsilon \neq 0$ and the strictly plurisubharmonic
function is given by $||z||^2$.
\end{proof}

When we move to homotopy $(2q{+}1)$-spheres mapping non-trivially to
${\rm Coker}(J_{2q{+}1})$, {\em there is no known example admitting a
  Stein filling.}  The following proposition, which is a more precise version of
Theorem~\ref{thm:Not_Stein_sphere} from the introduction, is a consequence of
Theorem \ref{thm:h-principle} as well as results of Wall and Schultz about
homotopy spheres bounding highly-connected manifolds.

\begin{Theorem} \label{thm:homotopy_sphere}
Let $\Sigma^{2q{+}1}$ be a homotopy sphere which maps non-trivially into ${\rm Coker}(J_{2q{+}1})$.  

\begin{enumerate}
	\item If $q \not\equiv  1, 3, 7$~mod~$8$ or if $q \equiv 1$~mod~$8$ and $q > 9$ or if $q = 7$ or $15$, then $\Sigma$ is not Stein fillable.
	\item If $q = 9$ or if $q \equiv 3, 7$~mod~8, then there is a
          cyclic subgroup $C^U_{q}\subset {\rm Coker}(J_{2q{+}1})$ such
          that $\Sigma$ is Stein fillable if and only if $\Sigma$ maps
          to zero in ${\rm Coker}(J_{2q{+}1})/C^U_{q}$.
\begin{enumerate}
	\item For $q = 9$, we have $C_9^U \cong 0$ or $\Z_2$.
	\item For $q \equiv 7$~mod~$8$, we have $C^U_{8k-1} \subset 4 \cdot {\rm Coker}(J_{16k-1})$.
\end{enumerate}
\end{enumerate}
\end{Theorem}

\noindent
There are many cases where the group $\Coker(J_{2q{+}1})/C_q^U$ is non-zero: we discuss
some examples in Corollary~\ref{lem:alpha_spheres} and Lemma \ref{lem:ker_alpha} below.
By the Generalized Poincar\'{e} Conjecture, Theorem \ref{thm:KM} and Theorem \ref{thm:homotopy_sphere}
imply the following

\begin{Corollary}\label{cor:smooth}
In general, the existence of a Stein fillable contact structure depends on the smooth structure 
of $M$ and not simply the underlying homeomorphism type of $M$. \qed
\end{Corollary}

\begin{proof}[Proof of Theorem \ref{thm:homotopy_sphere}]
Let $(W, J)$ be a Stein filling of $\Sigma^{2q{+}1}$.  Since $W$ has
handles only in dimension $(q{+}1)$ or less, it follows that $W$ is
obtained from $\Sigma$ by attaching handles of dimension $(q{+}1)$ or
greater.  Hence $W$ is $q$-connected and so
$\Sigma$ bounds a $q$-connected smooth manifold $W$ with a stable
complex structure $\bscxs_W$.  This constrains the diffeomorphism type
of $\Sigma$ as recorded in the statement of the proposition, as we now
explain.
  
The classification of oriented $q$-connected $(2q{+}2)$-manifolds with
boundary a homotopy sphere is given in \cite{Wall62}.  Such manifolds
are homotopy equivalent to a finite wedge of $(q{+}1)$-spheres and are
classified by triples
\[  (H, \lambda, \alpha) = (H_{q{+}1}(W), \lambda_W, \alpha_W) \]
where $(H_{q{+}1}(W), \lambda_W)$ is the usual intersection form of $W$,
which is a unimodular
bilinear form over the
integers, and $\alpha_W \colon H_{q{+}1}(W) \to \pi_{q}(SO(q{+}1))$ is a
quadratic refinement of $\lambda_W$ as explained in 
\cite[Lemma 2]{Wall62}.  The stablisation of $\alpha_W$ is a homomorphism
\[ S \alpha_W \colon H_{q{+}1}(W) \to \pi_{q}(SO), \]
which describes the stable tangent bundle of $W$ along each $(q{+}1)$-sphere 
in the homotopy type of $W$.  
In particular, $W$ admits a complex structure if and only if 
\begin{equation} \label{eq:almost_complex_criterion}
\im (S \alpha_W) \subset \im \bigl( \pi_q(U) \to \pi_q(SO) \bigr). 
\end{equation}

To study the diffeomorphism type of the homotopy sphere $\Sigma = \del
W$, Wall \cite[Theorems 2 \& 3]{Wall62} defined the bordism group
\[   A^{\an{q{+}1}}_{2q{+}2} := \{ [W] \, | \, \text{$W^{}$ is $q$-connected and $\del W \cong \Sigma$}  \},  \]
the \emph{rel.\,boundary bordism group} of smooth oriented $q$-connected $(2q{+}2)$-manifolds with boundary a homotopy sphere.
(The notation is from \cite{Stolz85} and
a similar notation appears in \cite[\S 17]{Wall67}.)
In analogy, we define the bordism group
\[   A^{U\an{q{+}1}}_{2q{+}2} := \{ [W, \acxs] \, | \, \text{$W^{}$ is $q$-connected and $\del W \cong \Sigma$}  \},  \]
to be the \emph{rel.\,boundary bordism group of almost complex
$q$-connected $(2q{+}2)$-manifolds} with boundary a stably complex
homotopy sphere.  We consider the homomorphisms
\[ A^{U\an{q{+}1}}_{2q{+}2} \xra{~F~} A^{\an{q{+}1}}_{2q{+}2} \xra{~\del~} \Theta_{2q{+}1} \xra{~\eta~} \Coker(J_{2q{+}1})\]
where $F$ remembers only the orientation underlying an almost complex
structure, $\del$ is defined by taking the diffeomorphism type of the
bounding homotopy sphere, and $\eta$ is the homomorphism from Theorem
\ref{thm:KM}.  The above discussion shows that the group
\[ C_q^U := \im(\eta \circ \del \circ F) \subset \Coker(J_{2q{+}1})\]
is isomorphic to the group of Stein fillable homotopy spheres modulo $bP_{2q{+}2}$.

\dcedit{Let $P_{2q{+}2} \subset A^{U\an{q{+}1}}_{2q{+}2}$ denote the subgroup generated
by parallelisable manifolds so that $\del P_{2q{+}2} = bP_{2q{+}2}$.
When $q = 2k$ is even, $\pi_{2k}(U) = 0$ and every $2k$-connected almost complex $(4k+2)$-manifold 
$W$ is parallelisable.  Hence 
\[ C_{2k} = (\eta \circ \del \circ F)(A^{U\an{2k+1}}_{4k+2}) = \eta(P_{4k{+}2}) = 0,\]
and it remains to consider the case where $q$ is odd.
Wall \cite[Theorem 11]{Wall67} computed the group $A^{\an{q{+}1}}_{2q{+}2}$ 
by proving that it is isomorphic to a certain Witt group of quadratic
forms $(H, \lambda, \alpha)$ as above.   
The computation is based on \cite[Theorem 11]{Wall67} where certain Grothendieck groups which 
surject onto $A^{\an{q{+}1}}_{2q{+}2}$ were computed.
We do not go into the details but summarise the facts relevant for our proof.
By \cite[Theorem 4]{Wall62}, $\eta \circ \del = 0$ for $3 \leq q \leq 7$ and hence $C_q^U = 0$ in these dimensions.

We now assume that $q = 2k+1 \geq 9$: by \cite[Theorem 11]{Wall67} and \cite[Theorem 2]{Wall62}
there is an isomorphism 
\begin{equation} \label{eq:A_bordism}
 \Phi \colon A^{\an{q{+}1}}_{2q{+}2} \cong P_{2q{+}2} \oplus \bigl( \pi_q(SO) \otimes \pi_q(SO) \bigr), 
\quad [H, \lambda, \alpha] \mapsto 
\bigl( \sigma(H, \lambda), \chi^2 \bigr). 
\end{equation}
Here $\chi \in H$ is an element such that 
$\lambda(x, \chi) = S\alpha(x)$ for all $x \in H$, intrepreted mod~$2$ if $\pi_q(SO) \cong \Z_2$, and 
$\chi^2 = \lambda(\chi, \chi)$ and we do not define $\sigma(H, \lambda) \in bP_{2q{+}2}$
since $\del(P_{2q{+}2}) \subset bP_{2q{+}2}$ does not concern us.
%
Let $F_* \colon \pi_q(U) \to \pi_q(SO)$ be the homomorphism induced by
$U\to SO$.  From~\eqref{eq:A_bordism} and~\eqref{eq:almost_complex_criterion} above, we see
that 
there is an isomorphism}
\[ \Phi^U \colon F \bigl( A^{U\an{q{+}1}}_{2q{+}2} \bigr) \cong \dcedit{P_{2q{+}2}} \oplus \bigl( \im(F_*) \otimes \im(F_*) \bigr) .\]
It follows that $C_q^U$ is the zero group if $F_* = 0$.  Given our
knowledge of the homomorphism $\pi_q(U) \to \pi_q(SO)$ this occurs unless $q \equiv 1, 3, 7$~mod~$8$. 
When $q \equiv 1, 3, 7$~mod~$8$, we see that $C_q^U$ is the
cyclic group generated by the element
\[  \bigl( \eta \circ \del \circ (\Phi^{U})^{-1} \bigr) \bigl( 0, F_*(1) \otimes F_*(1) \bigr) \in \Coker(J_{2q{+}1}), \]
where $1 \in \pi_q(U)$ is a generator.  For $q \equiv 1, 3$~mod~$8$, $F_*$ is
onto. However, if $q = 8k{+}1 > 9$ Schultz
\cite[Corollary 3.2]{Schultz} states that $\eta \circ \del = 0$,
proving that $C_{8k{+}1}^U = 0$ if $8k{+}1 > 9$.  Finally, for $q \equiv
7$~mod~$8$, $F_*(\pi_q(U)) \subset \pi_q(SO)$ is a subgroup of index
two and so the bilinearity of the tensor product ensures that $C_q^U =
4 \cdot \im(\eta \circ \del)$.  When $q = 15$, we have $\Coker(J_{31})
\cong \Z_2^2$, \cite[Table A3.3]{Ravenel86}, and hence $C_{15}^U = 0$.
\end{proof}

We now give some examples of exotic spheres which are not Stein fillable.
Since every homotopy sphere has a unique spin structure, there is a homomorphism
$\omega^{Spin} \colon \Theta_{n} \to \Omega_{n}^{Spin}$,
given by mapping a homotopy sphere to it spin bordism class.  Recall now the
$\alpha$-invariant
\[ \alpha \colon \Omega_*^{Spin} \to KO_* \]
which is a ring homomorphism from spin bordism to real $K$-theory defined by 
taking the $KO$-valued index of the Dirac operator on a spin manifold, \cite[\S 4.2]{Hitchin}.
Composing $\alpha$ with $\omega^{Spin}$ we obtain the $\alpha$-invariant for homotopy spheres
\[ \alpha \colon \Theta_n \to \Omega_n^{Spin} \to KO_n. \]
It is known that in all dimensions $8k{+}1$, $k \geq 1$, there are exotic
spheres with non-trivial $\alpha$-invariant in $KO_{8k{+}1} \cong
\Z_2$. The existence of such spheres follows from theorems of Milnor
and Adams as is explained in \cite[p.\,44]{Hitchin}.  If
$\alpha(\Sigma) = 1$ then $\Sigma$ does not bound a spin manifold.  On
the other hand, a Stein filling of $\Sigma$ is $4k$-connected and in
particular admits a unique spin structure.  Hence we obtain an
alternative proof of the following special case of Theorem
\ref{thm:homotopy_sphere}.

\begin{Lemma} \label{lem:alpha_spheres}
If $\Sigma \in \Theta_{8k{+}1}$ has $\alpha(\Sigma) = 1 \in KO_{8k{+}1}$
then $\Sigma$ is not Stein fillable. \qed
\end{Lemma}

Next we show that taking connected sums with $\alpha$-invariant-$1$
homotopy spheres can often destroy the Stein fillability of more
general manifolds.  Since $\pi_{8k{+}1}(SO/U) = 0$, it follows that
every homotopy $(8k{+}1)$-sphere has a unique stable complex structure
$\scxs_\Sigma$.
Given a stably complex manifold $(M, \scxs)$, we shall
write $(M \sharp \Sigma, \scxs \sharp \scxs_\Sigma)$ for the stably complex manifold obtained by taking the
connected sum of the stably complex manifolds $(M, \scxs)$ and $(\Sigma, \scxs_\Sigma)$.

\begin{Proposition} \label{prop:general_contamination}
Let $(M, \acs)$ be a Stein fillable almost contact manifold of
dimension $8k{+}1$ with $\scxs := \nacs$ and
$c_1(\scxs) = 0$.  If $\Sigma$ is a homotopy $(8k{+}1)$-sphere with
$\alpha(\Sigma) = 1 \in KO_{8k{+}1}$, then the stably complex manifold
$(M \sharp \Sigma, \scxs \sharp \scxs_\Sigma)$ is not Stein fillable.
\end{Proposition}

\begin{proof}
Since $c_1(\scxs) = 0$, there is a lift of the normal complex structure $\scxs \colon M \to BU$
to $BSU$.  It follows that there is a map of stable complex bundles 
$F \colon (B^{4k-1}_{\scxs}, \eta^{4k-1}_{\scxs}) \to (BSU, \pi_{SU})$.  The bundle map $F$ induces a homomorphism
of bordism groups
\[ F_* \colon \Omega_{8k{+}1}(B^{4k-1}_{\scxs}; \eta^{4k-1}_{\scxs}) \to \Omega_{8k{+}1}^{SU}.\]
Since $(M, \acs)$ is Stein fillable by Theorem \ref{thm:Stein}, every
$(4k-1)$-smoothing $\bscxs \colon M \to B^{4k-1}_{\zeta}$ is
null-bordant.  Moreover $\pi_{8k{+}1}(SO/U) = 0$ and thus $\Sigma$ admits a unique
$(B^{4k-1}_\scxs, \eta^{4k-1}_\scxs)$-structure $\bscxs_\Sigma$.
The connected
sum $(M \sharp \Sigma, \bscxs \sharp \bscxs_\Sigma)$ 
is a $(\scxs \sharp \scxs_\Sigma)$-compatible normal $(4k-1)$-smoothing in 
$(B^{4k-1}_{\scxs}, \eta^{4k-1}_\scxs)$.  Now we have
\[ F_*([M \sharp \Sigma, \bscxs \sharp \bscxs_\Sigma]) = [\Sigma, \scxs_\Sigma] \neq 0 \in \Omega_{8k{+}1}^{SU},\]
where the last inequality holds since the homomorphism $SU \to Spin$
induces a homomorphism $\Omega_*^{SU} \to \Omega_*^{Spin}$. 
Since $\alpha(\Sigma) = 1$, it follows that 
$[\Sigma, \scxs_\Sigma] \neq 0 \in \Omega_{8k{+}1}^{SU}$.
The above argument therefore shows that 
$[M \sharp \Sigma, \bscxs\sharp \bscxs_\Sigma] \neq 0 \in \Omega_{8k{+}1}(B^{4k-1}_{\scxs}; \eta^{4k-1}_{\scxs})$, 
and so by Theorem \ref{thm:Stein}, $(M \sharp \Sigma, \scxs \sharp \scxs_\Sigma)$ is not
Stein-fillable.
\end{proof}

We next construct a certain exotic $9$-sphere $\Sigma$ 
which lies inthe kernel of the $\alpha$ invariant
$\alpha \colon \Theta_9 \to KO_9$, but which does not bound a parallelisable manifold.
By Theorem \ref{thm:homotopy_sphere}, this homotopy sphere is not Stein fillable, but from 
a topological point of view, one can argue that it is one of the ``least exotic'' 
homotopy spheres which is not Stein fillable.  \dcedit{To the best of our knowledge, it is
not known whether $\Sigma$ admits a symplectically fillable contact structure.}

By Theorem \ref{thm:KM} above and results of Toda \cite[p.\,189]{Toda}, there is a short exact sequence
\[ 0 \to bP_{10} \to \Ker(\alpha) \to \Z_2 \to 0. \]
We shall given a explicit description of a homotopy sphere $\Sigma$ where 
$[\Sigma]$ generates $\Ker(\alpha)/bP_{10}$.
We first recall the well-known plumbing pairing
\[ \sigma_{p, q} \colon \pi_p(SO(q)) \times \pi_q(SO(p)) \longmapsto  \Theta_{p+q{+}1},
\quad (\beta, \gamma) \longmapsto  \del W(S(\beta), S(\gamma)),    \]
where $S \colon \pi_p(SO(q)) \to \pi_p(SO(q{+}1))$ is the the stabilisation homomorphism
and 
\[ W(S(\beta), S(\gamma)) : = (D^{q{+}1} \tilde \times_{S(\beta)} S^{p+1}) \cup_{D^{q{+}1} \times D^{p+1}} 
(D^{p+1} \tilde \times_{S(\gamma)} S^{q{+}1}) \]
is the compact smooth $(p+q{+}2)$-manifold obtained by plumbing the disc bundles of $S(\beta)$ and $S(\gamma)$
together: see for example \cite[Remark p.\,741]{Schultz}.
We let $\beta_5 \in \pi_3(SO(5)) \cong \Z$ and $\gamma_3 \in \pi_5(SO(3)) \cong \Z_2$ be 
generators and define the homotopy $9$-sphere
\[ \Sigma^9_{\beta_5, \gamma_3} : = \sigma_{3, 5}(\beta_5, \gamma_3). \]
Notice that there is a homotopy equivalence
$W(S(\beta_5), S(\gamma_3)) \simeq S^4 \vee S^6$, so that the manifold $W(S(\beta_5), S(\gamma_3))$ cannot admit a Stein structure, 
but from the point of view of the dimensions of the handles, $W(S(\beta_5), S(\gamma_3))$ is as close as possible 
to admitting a Stein structure.

\begin{Lemma} \label{lem:ker_alpha}
The homotopy $9$-sphere $\Sigma_{\beta_5, \gamma_3}$ maps to a generator of $\Ker(\alpha)/bP_{10} \cong \Z_2$.
\end{Lemma}

\begin{proof}
The proof starts with the exotic $8$-sphere $\Sigma^8 \in \Theta_8
\cong \Z_2$.  By \cite[Satz 12.1]{Stolz85} and \cite[Proposition
  12.20]{Husemoller94}, there is a diffeomorphism $\Sigma^8 \cong \del
W(\beta_5, \delta_4)$ where $\beta_5 \in \pi_3(SO(5))$ is as above and
$\delta_4 \in \pi_{4}(SO(4))$ is given by the composition $\tau_{S^4}
\circ \eta_3 \colon S^4 \to S^3 \to SO(4)$, where $\tau_{S^4}$ is the
characteristic map of the tangent bundle of the $4$-sphere and
$\eta_3$ is essential.  We claim that $\delta_4 = S(\delta_3)$ where
$\delta_3 \in \pi_{4}(SO(3)) \cong \Z_2$ is a generator.  To see this,
we use the commutative diagram of exact sequences
\[ \xymatrix{ \pi_3(SO(3)) \ar[d]^{\circ \eta_3} \ar[r]^S & \pi_3(SO(4)) \ar[d]^{\circ \eta_3} \ar[r]^{E_3} & \pi_3(S^3) \ar[d]^{\circ \eta_3} \\
\pi_{4}(SO(3)) \ar[r]^S & \pi_{4}(SO(4)) \ar[r]^{E_4} & \pi_{4}(S^3),  }   \]
where the horizontal sequences are part of the homotopy long exact sequence of the fibration
$SO(3) \to SO(4) \to S^3$, the vertical maps are given by pre-composition with 
$\eta_3$, and the map $E_3$ takes the Euler class of the corresponding bundle.  
Since $E_3(\tau_{S^4}) = \pm 2 \in \pi_3(S^3) \cong \Z$, 
it follows that $E_3(\tau_{S^4}) \circ \eta_3 = 0$ and so $E_4(\tau_{S^4} \circ \eta_3) = 0$.
Hence $\tau_{S^4} \circ \eta_3 \in \im(S)$.  Since $\Sigma^8$ is non-standard, $\tau_{S^4} \circ \eta_3$ is
non-zero and this proves the claim.
It follows that $\Sigma^8 \cong \sigma_{3, 4}(\beta_4, \delta_3)$, where $\beta_4 \in \pi_3(SO(4))$ stabilises to $\beta_5$.

To relate $\Sigma^8$ to $\Sigma_{\beta_5, \gamma_3}$ we shall use the Milnor-Munkres-Novikov pairing 
\cite[p.\,583]{Lashof63a},
\[ \tau_{p, q} \colon \pi_p(SO_q) \times \Theta_q \longmapsto \Theta_{p+q}, 
\quad (\alpha, \Sigma) \longmapsto \del W(\alpha, \Sigma), \]
where $W(\alpha, \Sigma)$ is the plumbing manifold
\[  (D^{q} \tilde \times_{\alpha} S^{p+1}) \cup_{D^{p+1} \times D^{q}} 
(D^{p+1} \times \Sigma^{q})  \]
obtained by plumbing the disc bundle of $\alpha$ with the trivial $(p+1)$-disc bundle over the homotopy
sphere $\Sigma$.
By \cite[Theorem 2.5]{Schultz}, if $\mu_n \in \pi_1(SO(n)) \cong \Z_2$ is a generator for $n \geq 3$, then
\[ \tau_{1,8}(\mu_8, \sigma_{3, 4}(\beta_4, \delta_3))) = \sigma_{3, 5}(S\beta_4, \delta_3 \circ \eta_4),  \]
so long as the Samelson product $S(\beta_4) \ast S(\mu_4) \in
\pi_{4}(SO(5))$ is trivial: we assume this for now and complete the proof.  Since
$\gamma_3 = \delta_3 \circ \eta_4$, it follows that $\Sigma_{\beta_5,
  \gamma_3} \cong \tau_{1, 8}(\mu_8, \Sigma^8)$.  But it is clear from
the definition of the pairing $\tau_{p, q}$ that $\eta(\tau_{1,
  8}(\mu_8, \Sigma_8)) = [\eta(\Sigma^8) \circ \eta_8] \in
\Coker(J_9)$.  But by \cite[p.\,189]{Toda}, $[\eta(\Sigma^8) \circ
  \eta_8] \neq 0 \in \Coker(J_9)$ and so $\Sigma_{\beta_5, \gamma_3}$
does not belong to $bP_{10}$.  On the other hand, $\Sigma_{\beta_5,
  \gamma_3}$ bounds a spin manifold by construction and so
$\Sigma_{\beta_5, \beta_3} \in \Ker(\alpha)$.

To complete the proof, we must show that the Samelson product $\beta_5 \ast \mu_5$ vanishes.
It suffices to show that $\beta_4 \ast \mu_4$ vanishes.
Recall that the Samelson product $\beta_4 \ast \mu_4 \colon S^4 \to SO(4)$ is defined 
to be the homotopy class of the map induced on $S^4$ by the following map
\[ S^3 \times S^1 \to SO(4), \quad (x, \lambda) \mapsto \beta_4(x)\mu_4(y)\beta_4^{-1}(x)\mu_4^{-1}(y). \]
Now, we represent $\beta_4$ and $\mu_4$ by the following maps:
\[ \beta_4(x)(y) = x \cdot y\quad \text{and} \quad \mu_4(\lambda)(y) = y \cdot \lambda ,\]
where $y \in \mathbb{H}$ is a quaternion, $x \in S^3$ a unit quaternion and
$\lambda \in S^1 \subset S^3$ a unit complex number.  Evidently
$\beta_4(x), \mu_4(\lambda) \in SO(4)$ commute for all values of $(x,
\lambda)$ and hence the Samelson product $\beta_4 \ast \mu_4$
vanishes.
\end{proof}

By Lemma \ref{prop:bP=Stein} homotopy spheres in $bP_{2q{+}2}\subset \Theta_{2q{+}1}$
(i.e. the ones mapping trivially to $\Coker(J_{2q{+}1}) $) are all Stein 
fillable, while Theorem \ref{thm:homotopy_sphere} 
shows that many homotopy spheres with non-trivial image in $\Coker(J_{2q{+}1}) $ do not 
admit Stein fillings. This observation naturally leads us to the following

\begin{Conjecture} \label{conj:stein_iff_bP}
A homotopy sphere $\Sigma ^{2q{+}1}$ is Stein fillable if and only $\Sigma^{2q{+}1} \in bP_{2q{+}2}$.
That is, in the notation of Theorem \ref{thm:homotopy_sphere}, $C_q^U = 0$ for all $q$.
\end{Conjecture}

Notice that while Theorem \ref{thm:homotopy_sphere} shows that many exotic spheres
are not Stein fillable, those same homotopy spheres might admit symplectically fillable contact structures.

\begin{Problem}[Symplectic fillability of homotopy spheres] \label{prob:symplectically-fillable-homotopy-spheres}
Do all homotopy spheres admit symplectically fillable contact
structures? If not, then determine all those that do.
\end{Problem}
The positive resolution of this problem would imply that symplectic fillability is invariant under the
action of the group of exotic spheres under connect sum. 
Notice that although our Filling Theorem~\ref{thm:Stein} is not useful in searching for 
symplectic fillings which are not also Stein fillings, Corollary \ref{cor:alpha-bordism}
may be helpful in finding symplectically fillable contact structures on homotopy spheres
which do not admit Stein fillings.

\section{Further properties of Stein fillable manifolds}
\label{sec:further}
In this section we discuss several topological properties of Stein fillable
manifolds.


\subsection{(Co)homological obstructions to Stein fillability} 
\label{subsec:topological_obstructions_to_stein_fillability}

In this subsection we discuss topological obstructions to Stein
fillability, which are not present in dimension 3, and some of their
consequences.  (See also \cite{PP} and \cite{EKP} for similar
obstructions.)  As usual, let $(M, \acs)$ be an almost contact
manifold with associated stable complex structure $\scxs$,
let $(B^{q-1}_{\scxs}, \eta^{q-1}_{\scxs})$ be the complex normal 
$(q-1)$-type of $(M, \scxs)$ and let $\bscxs \colon M \to B^{q-1}_\scxs$
be a $\scxs$-compatible normal $(q-1)$-smoothing.
We begin by observing that there is a commutative diagram,
\[ \xymatrix{& & B^{q{-}1}_{\scxs} \ar[d]^{p_{B} \times \eta^{q-1}_\scxs} \ar[drr]^{\eta^{q-1}_\scxs} \\
M \ar[urr]^{\bnacs} \ar[rr]_-{p_M \times \scxs} &  & P_{q{-}1}(M) \times BU
\ar[rr]_-{{\rm pr}_{BU}} & & BU, } \]
where $P_{q{-}1}(B) \simeq P_{q{-}1}(M)$ is the $(q{-}1)^{st}$
Postnikov stage of $M$ and $B$, and the maps $p_M \colon M \to
P_{q{-}1}(M)$ and $p_B \colon B \to P_{q{-}1}(M)$ are $q$-equivalences. 
We see that the induced homomorphism
\[ (p_B \times \eta^{q-1}_\scxs)_* \colon \Omega_{2q{+}1}(B^{q-1}_\scxs; \eta^{q-1}_\scxs) 
\to \Omega_{2q{+}1}^U(P_{q{-}1}(M)) \]
is such that $(p_B \times \eta^{q-1}_\scxs)_*([M, \bscxs]) = [(M, \scxs), p_M]$.
Applying Theorem \ref{thm:Stein} we obtain
\begin{Lemma} \label{lem:Pq-1(M)}
If  $[(M, \scxs), p_M] \neq 0 \in \Omega^U_{2q{+}1}(P_{q{-}1}(M))$,
then $(M, \scxs)$ does not admit a Stein fillable contact structure. \qed
\end{Lemma}

The following proposition combines Lemma \ref{lem:Pq-1(M)} with other elementary
observations to give obstructions to Stein fillability.  Let $\pi = \pi_1(M)$ denote
the fundamental group of $M$.

\begin{Proposition} \label{prop:fill1}
Suppose that $(M, \acs)$ is an almost contact manifold of dimension
$2q{+}1 \geq 5$ that admits a Stein fillable contact structure and let $u
\colon M \to K(\pi, 1)$ be the classifying map of the universal
cover of $M$.  Then the following hold:
\begin{enumerate}
\item \label{prop:fill1-fc} The homomorphism $u_* \colon H_i(M; \Z)
  \to H_i(K(\pi, 1))$ vanishes for $q{+}2 \leq i \leq 2q{+}1$.
  In particular $u_*([M]) = 0 \in H_{2q{+}1}(K(\pi, 1))$\ , where $[M]$
  denotes the fundamental class of $M$.
\item $M$ is not aspherical.
\item \label{prop:fill1-charc} For any $\beta \in H^j(P_{q{-}1}(M))$
  and for any $k$-tuple $\{i_1, \dots ,i_k \}$ of positive integers
  with $ 2q{+}1-j = 2\left(\Sigma _{n=1}^k \ i_n\right)$, all products
  of the form $p_M^*(\beta) \cup c_{i_1}(\scxs) \cup \dots \cup
  c_{i_k}(\scxs) \in H^{2q{+}1}(M)$ vanish.
\item \label{prop:fill1-chernc} For any $k$-tuple $\{i_1, \dots ,i_k
  \}$ of positive integers with $\Sigma _{n=1}^k\ i_n = q$, all
  products of Chern classes $c_{i_1}(\scxs) \cup \dots \cup
  c_{i_k}(\scxs) \in H^{2q}(M)$ vanish.
\end{enumerate}
\end{Proposition}

\begin{proof}
(1) Let $(W, \bar \scxs_W)$ be a $B$-nullbordism of $(M, \bar
\scxs)$.  After surgery we may assume that $W$
has no handles in dimension greater than $q{+}1$ and hence $H_{i}(W; \Z)
= 0$ for $i > q{+}1$.  Now $\bar \nu \colon M \to B$ factors over $W$
and $u \colon M \to K(\pi, 1)$ can be factored as $u_B \circ \bar \nu
\colon M \to B \to K(\pi, 1)$, where $u_B \colon B \to K(\pi, 1)$
classifies the universal covering of $B$.

(2) If $M$ is aspherical then $M \simeq K(\pi, 1)$ and so $u_*([M])$ is
a generator of the group $H_{2q{+}1}(K(\pi, 1)) \cong \Z$.  Now apply part (1).

(3) The integer 
\[ \an{p_M^*(\beta) \cup c_{i_1}(\scxs) \cup \thinspace\dots
\cup \thinspace c_{i_k}(\scxs), [M]}  \]
is an invariant of unitary bordism of $P_{q-1}(M)$.  By 
Lemma \ref{lem:Pq-1(M)} this integer vanishes for the unitary 
$P_{q-1}(M)$-manifold $((M, \scxs), p_M)$.  Since $H^{2q{+}1}(M) \cong \Z$,
this finishes the proof.

(4) We apply part (3) with $\beta \in H^1(P_{q{-}1}(M)) = H^1(M)$ and
then use a version of part (1) with mod $\Z/p$ coefficients to
conclude that $c_{i_1}(\scxs) \cup \dots \cup c_{i_k}(\scxs)$ vanishes
in $H^{2q}(M; \Q)$ and also  in $H^{2q}(M; \Z/p)$ for all primes $p$.  It
follows that $c_{i_1}(\scxs) \cup \dots \cup c_{i_k}(\scxs) = 0 \in
H^{2q}(M)$.
\end{proof}
Proposition~\ref{prop:fill1} allows us to prove the following

\begin{Corollary} \label{cor:choice_of_acs} 
In general, the Stein fillability of an almost
  contact manifold $(M, \acs)$ depends on the choice $\acs$ and not
  just the underlying diffeomorphism type of $M$.
\end{Corollary}

\begin{proof} 
  The manifold $M = S^1 \times S^6$ clearly admits a Stein fillable
  almost contact structure $\varphi_0$ since $M = \del (S^1 \times
  D^7)$.  On the other hand, $S^6$ admits an almost complex structure
  $\acxs$ with $c_3(\acxs) = 2 \in H^6(S^6)$.  For the induced almost
  contact structure $\acs _1$ on $M$, Proposition \ref{prop:fill1}
  \eqref{prop:fill1-chernc} implies that $(M, \acs_1)$ is not Stein
  fillable.
\end{proof}

As a consequence of Proposition \ref{prop:fill1}, we obtain
obstructions to the Stein fillability of certain Boothby-Wang contact
structures.

\begin{Example}[Boothby-Wang contact structures]\label{ex:Boothby_Wang}
A \emph{Boothby-Wang contact structure} on a (nontrivial) principal
$S^1$-bundle
$$S^1 \longrightarrow E \stackrel{\pi} \longrightarrow B$$ over a
symplectic base $(B, \omega )$ of dimension $2q$ with
$c_1(E)=[\frac{\omega}{2\pi }]$ is given as the kernel of an
$S^1$-invariant $1$-form $\alpha$ which is non-vanishing on the fibers
and satisfies $ d\alpha = \pi^*(\omega)$ for some integral symplectic
form $\omega$.
\end{Example}
Note that the associated disc bundle of the principal $S^1$-bundle $E$
is a strong symplectic filling (see, e.g., \cite{Geiges&Stipsicz10},
Lemma 3), which is not Stein since it is homotopy equivalent to the
$2q$-dimensional base $B$. (However, if the base is $\mathbb{C}P^2$
and the Euler class of the bundle is a generator of
$H^2(\mathbb{C}P^2; {\mathbb {Z}})$, then the total space is the
$5$-sphere which is of course Stein fillable.) On the other hand we do
have the following example:

\begin{Example}[Lens spaces] \label{ex:L5}
Let $L^5_k$ be the standard $5$-dimensional lens space with cyclic
fundamental group of order $k$.  That is, $L^5_k$ is the quotient of 
\[ S^5=\{ (z_1,
z_2, z_3)\in \C ^3\mid \vert z_1\vert ^2+\vert z_2\vert ^2+\vert
z_3\vert ^2=1\}, \] 
with the action of a generator of $\Z_k$ defined by 
\[
(z_1, z_2, z_3) \mapsto (\mu z_1, \mu z_2, \mu z_3)
\]
for $\mu \in \C$ a $k^{th}$ root of unity.  The resulting manifold inherits
an $S^1$-bundle projection $\pi \colon L^5_k \to \C P^2$.  
Since the classifying map of the universal
cover $u$ induces a non-trivial map
$$u_* \colon H_i(L^5_k) \to H_i(K(\Z_k, 1)),$$ we conclude that
although the lens spaces $L^5_k$ are symplectically fillable (by the
Boothby-Wang construction), by Lemma \ref{prop:fill1} \eqref{prop:fill1-fc} they are not Stein
fillable for all $k \geq 2$. (Obstructions for Stein fillability of
these manifolds were already noticed in \cite{EKP}, cf. also
\cite{PP}.)
\end{Example}

In conclusion, we see both examples of Boothby-Wang contact structures which are
Stein fillable, and others which are not. This observation leads to the following question:

\begin{Problem}[Fillability of Boothby-Wang manifolds]
Determine which Boothby-Wang manifolds are Stein/exactly fillable.
\end{Problem}

\noindent By recent work of Massot, Niederkr\"uger and Wendl
\cite{Massot12}, Proposition \ref{prop:fill1} also gives examples of
exactly fillable contact structures that are not Stein fillable in all
dimensions. Such examples were discussed in \cite{Bowden12} for
3-dimensional manifolds, although in this case the non-fillability
only applied to certain contact structures rather than to the
manifolds themselves. We are now in the position to provide the proof of 
Theorem~\ref{thm:exact-notStein} from the Introduction:

\begin{proof}[Proof of Theorem~\ref{thm:exact-notStein}]
By \cite[Theorem~C]{Massot12}, there are exact symplectic fillings 
of the form $M \times [0,1]$ such that both ends are convex in all dimensions $2q{+}2$. The
manifolds $M$ are quotients of contractible Lie groups and are
consequently aspherical. 
 After attaching a Weinstein $1$-handle to $M \times [0, 1]$, 
we obtain an exact filling of $N = -M \# M$.  Assuming that $q > 1$, $\pi_1(N)$ is
the free product two copies of $\pi_1(M)$ and so there is a homotopy equivalence
$K(\pi_1(N), 1) \simeq K(\pi_1(M), 1) \vee K(\pi_1(M), 1)$.
Since $M$ is aspherical, we see that the classifying map of the universal cover of $N$ 
maps non-trivially on $H_{2q{+}1}(N)$. Hence by Proposition~\ref{prop:fill1} \eqref{prop:fill1-fc}, $N$ is not
Stein fillable if $q > 1$.
\end{proof}

\subsection{Stein fillability and orientations}  \label{subsec:orientation}
A cooriented contact structure $\xi =
\ker (\alpha)$ determines an orientation of the underlying $(2q{+}1)$-manifold $M$,
since the form $\alpha \wedge (d\alpha)^q$ is nowhere vanishing. 
When we speak of an oriented manifold admitting a contact structure,
we mean that the orientation determined by the contact structure is
the given one. Moreover, if the dimension of $M$ is of the form
$4k{+}1$, and hence the dimension of the Stein filling of $M$ is of the
form $4k+2$, then taking the conjugate complex structure on $W$
reverses orientations. The resulting Stein fillable contact
structure then gives the opposite coorientation of $\xi$,
i.e. replaces $\alpha$ by $-\alpha$, which in turn swaps the orientation
determined by the contact structure. So in these dimensions it is clear that $M$ is
Stein fillable if and only if $-M$ is.

\dcedit{
However, if the dimension of $M$ is $4k+3$, then it is not immediately clear
that $M$ is Stein fillable if and only if $-M$ is Stein fillable:
indeed the statement is false in dimension $3$, with many examples given by
Seifert fibred spaces, the most famous of which is the Poincar\'{e}
homology sphere \cite{Lisca}.} On the other hand, Eliashberg's $h$-principle implies the following

\begin{Proposition} \label{prop:orientation}
Let $(M, \acs)$ be an almost contact $(2q + 1)$-dimensional manifold with $q \geq 2$ and associated
stable complex strcovucture $\scxs$.  Then $(M, \scxs)$ is Stein fillable if and only if $(-M, -\scxs)$ is.
\end{Proposition}

\begin{proof}
The fact that any Stein filling $W$ of $(M, \scxs)$ is a manifold with boundary means that $TW$ admits a nonvanishing section and thus as complex bundles
$$(TW,J) \cong (E,J|_E) \oplus \underline{\mathbb{C}}.$$
We then define an almost complex structure $\bar{J}$ by taking $J|_E$ on $E$ and the conjugate complex structure on $\underline{\mathbb{C}}$. The almost complex structure $\bar{J}$ then 
induces the orientation $-W$, and applying Eliashberg's $h$-principle gives a Stein fillable contact structure on $-M$ with associated stable complex structure $-\scxs$. 
\end{proof}

\section{Subcritical Stein fillings and Stein fillings of products} \label{sec:sub-critical}

We fix a closed almost contact $(2q{+}1)$-manifold $(M, \acs)$ and as
usual we let $\scxs = \nacs$ denote the stable complex structure
induced by the almost contact structure $\acs$.  A {\em subcritical
  Stein filling} of $(M, \acs)$ is a Stein filling $(W, J)$ of $(M,
\acs)$ where $W$ admits a handle decomposition with handles of
dimension $q$ and less.  Subcritical Stein fillings have special
properties; see \cite{Cieliebak&Eliashberg12}.

Another filling question is the following: suppose that $(F, \acxs_F)$
is an almost complex structure on a closed, \dcedit{oriented}
surface $F$.  Then we can ask if the product almost contact manifold
$(M \times F, \acs \times \acxs_F)$ admits a Stein filling.  It is
easy to see that if $(M,\acs)$ has a subcritical Stein filling, then
$(M \times F, \acs \times \acxs_F)$ is Stein fillable: if $(W,J_W)$ is
the subcritical filling of $(M,\acs)$ then $(W\times F, J_W\times
J_F)$ is an almost complex manifold with boundary $(M\times F, \acs
\times J_F)$ which admits a handle decomposition with handles of
dimension $q{+}2$ and less (and the dimension of $W\times F$ is $2q+4$),
therefore Eliashberg's $h$-principle implies the result.

We shall further relate the two questions about Stein fillings to the
bordism theory of $(B^q_\scxs, \eta^q_\scxs)$, the complex normal
$q$-type of $(M, \scxs)$.  
We pose five related questions:
\begin{enumerate}
\item[(A)] When does $(M, \acs)$ admit a subcritical Stein filling?
\item[(B)] When does $(M \times F, \acs \times \acxs_F)$ admit a Stein filling?
\item[(C)] When does $[M, \bscxs] = 0 \in \Omega_{2q{+}1}(B^q_\scxs; \eta^q_\scxs)$ hold?
\item[(D)] When does $\bscxs_*([M]) = 0 \in H_{2q{+}1}(B^q_\scxs)$ hold?
\item[(E)] When does $TH_q(M)$, the torsion subgroup of $H_q(M)$, vanish?
\end{enumerate}

We next graphically summarise the relationship between positive
answers to the questions above, writing $g(F) > 0$ for the case
where $F$ has positive genus; see Theorem
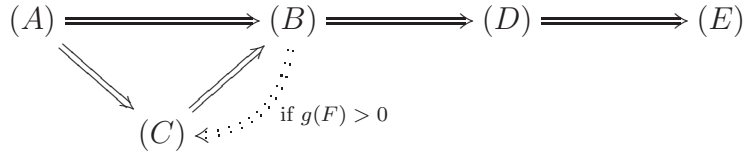
\ref{thm:sub-c_and_products} below.
\[ \xymatrix{ (A) \ar@2{->}[dr] \ar@2{->}[rr] && (B) \ar@2{->}[rr]
\ar@2{->}@/^1.5pc/@{.>}[dl]^(0.4){\text{if $g(F)>0$}} & & (D) \ar@2{->}[rr] & & (E) \\ & (C)
\ar@2{->}[ur] } \]

\begin{Theorem}[Subcritical Filling Theorem] \label{thm:sub-c_and_products}
Let $(B^q_\scxs, \eta^q_\scxs)$ be the complex normal $q$-type of $(M, \scxs)$
and let $\bscxs \colon M \to B^q_\scxs$ be any $\scxs$-compatible normal $q$-smoothing.
If $q \geq 2$, then the following hold.
\begin{enumerate}
\item If $(M, \acs)$ admits a subcritical filling then $[M, \bscxs] = 0
  \in \Omega_{2q{+}1}(B^q_\scxs ; \eta^q_\scxs)$. 
\item If $[M, \bscxs] = 0 \in \Omega_{2q{+}1}(B^q_\scxs; \eta^q_\scxs)$
  then $(M \times F, \acs \times \acxs_F)$ admits a Stein filling.
\item If $(M \times F, \acs \times \acxs_F)$ admits a Stein filling
  and $g(F) > 0$, then the bordism class $[M, \bscxs]$ satisfies
 $[M, \bscxs] = 0 \in \Omega_{2q{+}1}(B^q_\scxs; \eta^q_\scxs)$.  
 In particular, $(M, \zeta)$ is Stein fillable.
\item If $(M \times F, \acs \times \acxs_F)$ admits a Stein filling
then $\bscxs_*([M]) = 0 \in H_{2q{+}1}(B^q_\scxs)$.
\item If $\bscxs_*([M]) = 0 \in H_{2q{+}1}(B^q_\scxs)$ then $TH_q(M) = 0$.
\end{enumerate}
\end{Theorem}

\begin{proof}
(1) The proof is similar to the proof of part  (1) of Lemma \ref{lem:topological_filling}.
Let $(W, \scxs_W)$ denote the subcritical filling with its induced
stable complex structure; it is built from $(M, \scxs)$ by
  adding handles with stable complex structure of dimension $q{+}2$ and higher. 
  Therefore the complex normal $q$-type of $M$ can be identified with that of $W$ 
  and the claim follows.

(2) Let $\scxs_F$ be the stable normal complex structure defined by
$\acxs_F$, let $P_q(F)$ be the $q^{th}$ Postnikov stage of $F$ and 
let $p_F \colon F \to P_q(F)$ be a $(q{+}1)$-equivalence (if $g(F)>0$,
then $P_q(F) = K(\pi_1(F), 1)$), and let $L_{\scxs_F}$ be the unique complex
line bundle over $P_q(F)$ such that $c_1(\scxs_F) = p_F^*(c_1(L_{\scxs_F}))$.
The complex normal $q$-type of $(M \times F, \scxs \times \scxs_F)$ is given by
\[ (B^q_{\scxs \times \scxs_F}, \eta^q_{\scxs \times \scxs_F}) =
(B^q_\scxs \times P_q(F), \eta^q_\scxs \oplus L_{\scxs_F}), \]
where, as in Section \ref{subsec:normal_k_types}, $\eta^q_\scxs \oplus
L_{\scxs_F}$ denotes the exterior Whitney sum of stable complex
bundles.
By assumption 
there is a $(B^q_\scxs, \eta^q_\scxs)$-null bordism $(W, \bscxs_W)$ of $(M,
\bscxs)$.  We observe that
\[ \bscxs_W \times p_F \colon W \times F
\to B^q_\scxs \times P_q(F) \]
is a $(B^q_{\scxs \times \scxs_F}, \eta^q_\scxs \times L_{\scxs_F})$-nullbordism of $(M \times F, \scxs \times \scxs_F)$.  
By Theorem \ref{thm:Stein}, $(M \times F, \acs \times \acxs_F)$ is Stein fillable.

(3) If $g(F)>0$, then $F$ is a $K(\pi, 1)$ manifold and $P_q(F) = F$.
It follows that the complex normal $q$-type of $(M \times F, \scxs \times \scxs_F)$
is given by
\[ (B^q_{\scxs \times \scxs_F}, \eta^q_{\scxs \times \acxs_F}) =
(B^q_\scxs \times F, \eta^q_\scxs \oplus L_{\scxs_F}), \]
where $L_{\scxs_F}$ is defined as in the proof of $(2)$. 
There is a canonical isomorphism of bordism groups
\[ \theta \colon \Omega_*(B^q_\scxs \times F; \eta^q_\scxs \oplus \scxs_F) 
~\cong~ \Omega_*^{(B^q_\scxs;\eta^q_\scxs)}(F; 
L_{\scxs_F}) \]
with range the $\scxs_F$-twisted $(B^q_\scxs, \eta^q_\scxs)$-bordism
group of $F$. Taking the transverse inverse image of a point $x \in
F$ defines a homomorphism
\[  \pitchfork \colon \Omega_*^{(B^q_\scxs ;\eta^q_\scxs)}(F; 
L_{\scxs_F}) \to \Omega _{*-2}(B^q_\scxs; \eta^q_\scxs). \]
On the other hand, taking the product with $(F, \scxs_F)$ defines a homomorphism
\[ \Pi \colon \Omega_{*}(B^q_\scxs ;\eta^q_\scxs) \to 
\Omega_{* + 2} (B^q_\scxs \times F;  \eta^q_\scxs \oplus \scxs_F) ,
\quad [X, \scxs_X] \mapsto [X \times F, \scxs_X \times \scxs_F]. \]
From the definitions of the above homomorphisms, we see that there is
a commutative diagram
\[ \xymatrix{  \Omega_{*}(B^q_\scxs; \eta^q_\scxs)  
\ar[d]^{\Pi} \ar@/^/[drr]^{\Id} \\
\Omega_{* + 2} (B^q_\scxs \times F; \eta^q_\scxs \oplus \scxs_F)
\ar[r]^(0.55)\theta & \Omega_{*+2}^{(B^q_\scxs;\eta^q_\scxs)}(F;
L_{\scxs_F}) \ar[r]^\pitchfork &
\Omega_{*}(B^q_\scxs; \eta^q_\scxs). } \]
If $(M \times F, \acs \times \acxs_F)$ is Stein fillable, then by
Theorem \ref{thm:Stein}, $[M \times F, \bscxs \times \bscxs_F] =
\Pi([M, \scxs]) = 0 \in \Omega_{2q+3}(B^q_\scxs \times F; \eta^q_\scxs
\oplus \scxs_F)$.  The diagram then shows that that $[M, \bscxs] = 0
\in \Omega_{2q{+}1}(B^q_\scxs ; \eta^q_\scxs)$.

(4) If $(M \times F, \scxs \times \scxs_F)$ is Stein fillable then by
Theorem \ref{thm:Stein} all $(\scxs \times \scxs_F)$-compatible normal
$q$-smoothings of $(M \times F, \scxs \times \scxs_F)$ bound over 
$(B^q_\scxs \times P_q(F), \eta^q_\scxs \times L_{\scxs_F})$.
As a consequence,
\[ (\bscxs \times p_F)_*([M \times F]) = 0 \in H_{2q+3}(B^q_\scxs \times P_q(F)). \]
Since $(p_F)_*([F]) \in H_2(P_q(F)) \cong \Z$ is a generator, 
the result now follows from the Kunneth theorem.

(5) Recall that the linking form of $M$ is a nonsingular bilinear pairing
\[ TH_q(M) \times TH_q(M) \to \Q/\Z. \]
We will show that the assumption $\bscxs_*([M]) = 0$ ensures that the linking
form of $M$ vanishes, and this can only happen if $TH_q(M)$ vanishes.

Let $p \colon H_q(M) \to TH_q(M)$ be a splitting and let 
$p \colon K(H_q(M), q) \to K(TH_q(M),q)$ also denote the induced map of Eilenberg-MacLane 
spaces.  The map $q_M \colon M \to K(H_q(M),q)$ inducing the identity on $H_q$ is
such that the composition 
\[  M \xra{~q_M~} K(H_q(M), q) \xra{~p~} K(TH_q(M), q) \]
satisfies $(p \circ q_M)_*([M]) \neq 0 \in H_{2q{+}1}(K(TH_q(M)), q)$ if $TH_q(M) \neq 0$:
This follows from the cohomological definition of the linking form and its nonsingularity.
But since the map $\bscxs \colon M \to B^q_\scxs$ is a $(q{+}1)$-equivalence, it follows that $q_M$
can be factored through $\bscxs$.  Hence if $\bscxs_*([M]) = 0$ then $(p \circ q_M)_*([M]) = 0$,
the linking form of $M$ vanishes, and $TH_q(M) = 0$.
\end{proof}

\begin{Example} \label{ex:sub-c_and_bordism}
The converse of (1) in Theorem~\ref{thm:sub-c_and_products} (that
$[M, \bscxs] = 0 \in \Omega_{2q{+}1}(B^q_\scxs ; \eta^q_\scxs)$ implies
that $(M,\acs )$ is subcritically Stein fillable) does not
hold. Notice first that the adaptation of the proof breaks down,
since the surgery method of \cite{Kreck99} works only up to the middle
dimension.  Indeed, if $\Sigma \in bP_{2q{+}2}$ is exotic, then $\Sigma$
admits an almost contact structure with stabilisation $\scxs$ such
that $[\Sigma, \scxs] = 0 \in \Omega_{2q{+}1}(B^q_\scxs; \eta^q_\scxs)$,
but $\Sigma$ does not admit a subcritical Stein filling: a subcritical
Stein filling of a homotopy sphere must be contractible,
implying that the filling is diffeomorphic to the disk and that the homotopy sphere is
standard.
\end{Example}

\begin{Example}
A simple example of a Stein fillable manifold $M$ with the property
that $M\times S^2$ is not Stein fillable is provided by $M=S^1\times
S^2\times S^2$: by the fact that $M=\partial (S^1\times S^2\times
D^3)$ we see that it is Stein fillable, while
Proposition~\ref{prop:fill1} implies that $S^1\times S^2\times
S^2\times S^2$ is not Stein fillable. 
\end{Example}

Theorem \ref{thm:sub-c_and_products} shows that the existence of a subcritical
filling of $(M, \acs)$ places strong constraints on the topology of $M$.  We next 
pursue this point further for simply connected manifolds in dimensions $5$ and $7$.
Let $S^3 \widetilde \times S^2$ and $S^5 \widetilde \times S^2$ be the total spaces of
the nontrivial linear $n$-sphere bundle over the $2$-sphere, $n = 3, 5$.

\begin{Proposition} \label{prop:sub-crit}
Suppose that $(M, \acs)$ is a simply connected almost contact manifold of 
dimension $5$ or $7$ and that $(M, \acs)$ admits a subcritical Stein filling.
\begin{enumerate}
\item If $M$ has dimension $5$, then there is a  nonnegative integer 
$r$ such that  $M$ is diffeomorphic to one of the connected sums
\[ \sharp_r(S^3 \times S^2) \quad \text{or} \quad ( S^3 \wt \times S^2) \sharp_r(S^3 \times S^2),  \]
depending on whether $M$ is spin or not.	
\item If $M$ has dimension $7$ and $\pi_2(M)$ is torsion free, 
then there are non-negative
integers $r,s$ such that $M$ is diffeomorphic to one of the connected sums
\[ \sharp_r(S^5 \times S^2) \sharp_s(S^4 \times S^3) \quad \text{or} \quad 
(S^5\wt \times S^2) \sharp_r(S^5 \times S^2) \sharp_s(S^4 \times S^3), \]
depending on whether $M$ is spin or not.	

\end{enumerate}
\end{Proposition}

\begin{proof}
If $(W, \acxs)$ is a subcritical filling of $(M, \acs)$, then $W$ is obtained from $M$
by attaching $(q{+}2)$-handles and higher.  It follows that the map $M \to W$ is a $(q{+}1)$-equivalence.
Now by Theorem \ref{thm:sub-c_and_products} (1), (3), (4) and (5), $TH_q(M) = 0$ and
so $TH_q(W) = 0$.  Since $W$ is also a simply connected manifold consisting only of handles of dimension
$q$ or less, we conclude the following: if $q = 2$, it follows the $W$ is homotopy equivalent to a wedge 
of $2$-spheres and if $q = 3$, then $W$ is homotopy equivalent to wedge of $2$-spheres and $3$-spheres.
Note that for the case $q = 3$ we use the assumption that $\pi_2(M) \cong H_2(M)$ is torsion free.
It follows, using the terminology of \cite{Wall67} that the manifold $W$ is then a stable thickening 
of a wedge of spheres.  By \cite[Propsition 5.1]{Wall67} stable thickenings are classified up to
diffeomorphism by their homotopy type and the map classifying their stable tangent bundle.
Now for the $W$ we consider, $[W, BSO] \cong H^2(W; \Z_2)$, the bijection being given by the second
Stiefel-Whitney class.  If $\natural$ denotes the boundary connected sum of manifolds with boundary
and $D^4 \wt \times S^2$ and $D^6 \wt \times S^2$ denote the non-trivial linear disc bundles over $S^2$,
we deduce that $W$ is diffeomorphic to one of the following manifolds:
\begin{align*} \text{Dimension 5:} & \quad \natural_r(D^4 \times S^2) \quad \ \ \  \ \ \ \ \ \ \ \ \  \ \ \ \,  
	\text{or} \quad ( D^4 \wt \times S^2) \natural_r(D^4 \times S^2),\\
 \text{Dimension 7:} & \quad \natural_r(D^6 \times S^2) \sharp_s(D^5 \times S^3) \quad  \text{or} \quad 
(D^6 \wt \times S^2) \sharp_r(D^6 \times S^2) \sharp_s(D^5 \times S^3),
\end{align*}
and the proposition follows.
\end{proof}

We conclude this section by viewing Example \ref{ex:sub-c_and_bordism} 
in a more general framework.
Let $L_{2q{+}2}^{s, \tau}(\pi)$ denote 
the group of units in the surgery obstruction monoid $l_{2q{+}2}(\pi)$,
which was defined in \cite[\S 6]{Kreck99}.
(The notation is from \cite[\S 4]{Kreck85} and differs from \cite{Kreck99}.
In addition, $L_{2q{+}2}^{s, \tau}(\pi)$ may be identified with the 
obstruction group $L_{2q{+}2}^C(\pi)$ of \cite[17D]{Wall99}.)
The group $L_{2q{+}2}^{s, \tau}(\pi)$
acts on the set of $(B^q_\scxs, \eta^q_\scxs)$-diffeomorphism classes of complex
normal $q$-smoothings $\bscxs \colon M \to B^q_\scxs$ without changing
the $(B^q_\scxs, \eta^q_\scxs)$-bordism class.  That is, writing
$(M_{+\rho}, \bscxs_{+\rho})$ for the action of 
$\rho \in L_{2q{+}2}^{s,\tau}(\pi)$ on $(M, \bscxs)$, we have
\[  [M, \bscxs] = [M_{+\rho}, \bscxs_{+\rho}] \in
\Omega_{2q{+}1}(B^q_\scxs ; \eta^q_\scxs) .\]
For example, if $M$ is simply connected, then $L_{2q{+}2}^{s, \tau}(e)
\cong \Z$ or $\Z_2$ as $q$ is odd or even, and the action of
$L_{2q{+}2}^{s, \tau}(e)$ is via connected sum with $(\Sigma, \bscxs_\Sigma)$, 
where $\Sigma$ is a generator of $bP_{2q{+}2}$ and $\bscxs_\Sigma$ is a certain a 
$(B^q_\scxs, \eta^q_\scxs)$-structure on $\Sigma$.
\begin{Question} \label{qn:sub-c_and_bP}
Suppose that $\bscxs \colon M \to B^q_\scxs$ is a normal
$q$-smoothing such that $[M, \bscxs] = 0 \in \Omega_{2q{+}1}(B^q_\scxs ;
\eta^q_\scxs)$.  Under what conditions on $M$ can we
deduce that there is an element $\rho \in L_{2q{+}2}^{s,\tau}(\pi)$ 
such that $(M_{+\rho}, \bscxs_{+\rho})$ admits a
subcritical Stein filling?  For example, if $M$ is simply connected,
is there a homtopy sphere $\Sigma \in bP_{2q{+}2}$ such that $(M \sharp
\Sigma, \bscxs \sharp \bscxs_\Sigma)$ admits a subcritical Stein
filling?  
\end{Question}

\end{document}